\documentclass[10pt,leqno]{article}

\usepackage{authblk}
\usepackage{amsmath,amsthm,amssymb,mathtools}
\usepackage{color,tikz-cd,hyperref,geometry}
\hypersetup{colorlinks=true,linkcolor=blue,citecolor=blue,urlcolor=blue,linktocpage}
\usetikzlibrary{patterns}
\numberwithin{equation}{section}

\newtheorem{theorem}{Theorem}[section]
\newtheorem{corollary}[theorem]{Corollary}
\newtheorem{lemma}[theorem]{Lemma}

\theoremstyle{definition}
\newtheorem{definition}[theorem]{Definition}

\theoremstyle{notation}

\theoremstyle{theorem}
\newtheorem{proposition}[theorem]{Proposition}

\theoremstyle{definition}

\theoremstyle{definition}
\newtheorem{remark}[theorem]{Remark}
\newtheorem*{remark*}{Remark}

\newcommand{\X}{X^{2n+2}_k}
\newcommand{\Y}{Y^{2n+2}_k}
\newcommand{\XX}{X^{2n+2}_{2k}}
\newcommand{\YY}{Y^{2n+2}_{2k}}
\newcommand{\Z}{Z^{2n+2}}
\newcommand{\cC}{\mathcal{C}}
\newcommand{\cD}{\mathcal{D}}
\newcommand{\cE}{\mathcal{E}}
\newcommand{\cW}{\textup{WFuk}}
\newcommand{\hocolim}{\textup{hocolim}}
\newcommand{\Cone}{\textup{Cone}}
\newcommand{\Mod}{\textup{Mod}}
\newcommand{\Hom}{\textup{Hom}}
\newcommand{\Tw}{\textup{Tw}}

\title{Exotic families of symplectic manifolds with Milnor fibers of $ADE$-type}
\author{Dongwook Choa \thanks{dwchoa@kias.re.kr}}
\author[$\dagger$]{Dogancan Karabas \thanks{dogancan.karabas@northwestern.edu }}
\author[$\ddagger$]{Sangjin Lee \thanks{sangjinlee@ibs.re.kr}}
\affil{Department of Mathmatics, Korea Institute for Advanced Study, Seoul 02455, Korea}
\affil[$\dagger$]{Department of Mathematics, Northwestern University, 2033 Sheridan Rd, Evanston, IL 60208, United States} 
\affil[$\ddagger$]{Center for Geometry and Physics, Institute for Basic Science (IBS), Pohang 37673, Korea}

\begin{document}

\maketitle

\noindent{\bf Abstract}

In this paper, we give infinitely many diffeomorphic families of different Weinstein manifolds. 
The diffeomorphic families consist of well-known Weinstein manifolds which are the Milnor fibers of $ADE$-type, and Weinstein manifolds constructed by taking the end connected sums of Milnor fibers of $A$-type.
In order to distinguish them as Weinstein manifolds, we study how to measure the number of connected components of wrapped Fukaya categories.
%



\tableofcontents

\section{Introduction}
\label{section introduction}
\subsection{Introduction}
The constructions of exotic Weinstein manifolds, i.e., diffeomorphic manifolds having different Weinstein structures, have been studied extensively.
See, for example, \cite{McLean} or \cite{Abouzaid-Seidel}.
In this paper, we investigate Weinstein manifolds which are exotic to the Milnor fibers of simple singularities. 

We construct different Weinstein manifolds by attaching the same Weinstein handles to the same Weinstein manifold in different ways. 
However, as smooth handles, they are attached in the same way. 
Thus, the construction gives diffeomorphic, but different Weinstein manifolds. 

We note that the idea has been used historically. 
See \cite{Maydanskiy}, \cite{Maydanskiy-Seidel}, and \cite{Abouzaid-Seidel}.
The main difference between the previous works and the current paper is that, for some particular Weinstein manifolds, we will identify {\em plumbings} and {\em end connected sums} procedures as the same smooth handle attachments.
However, they are different as Weinstein handle attachments.  
Thus, one can show that, for example, the Milnor fiber of $A_{k+1}$ singularity and the end connected sum of the Milnor fiber of $A_k$ singularity and a cotangent bundle of the sphere are diffeomorphic but different as Weinstein manifolds.  

What allows us to distinguish a plumbing and an end connected sum as Weinstein manifolds, is the {\em number of connected components of the wrapped Fukaya categories}. 
Since the end connected sum procedure increases the number of connected components, but the plumbing procedure does not, two procedures lead to the different Weinstein manifolds.  
Thus, we can distinguish different Weinstein manifolds without using vanishing wrapped Fukaya categories or without using symplectic homologies.
The first is used in \cite{Maydanskiy}, \cite{Maydanskiy-Seidel}, and the second is used in \cite{McLean}, \cite{Abouzaid-Seidel}. 

The present paper contains two generalizations of the above construction of exotic pairs. 
More specific results will be stated in Section \ref{subsection results}.

\subsection{Results}
\label{subsection results}
In Sections \ref{section construction} -- \ref{section diffeo}, we observe that a plumbing procedure and an end connected sum procedure can be realized as the same smooth handle attachments, but different Weinstein handle attachments. 

More precisely, we prove Theorem \ref{thm rough statement}.
\begin{theorem}[Technical statement is Theorem \ref{thm main}]
	\label{thm rough statement}
	Let $n$ be an odd integer.
	The Milnor fiber of $A_{k+1}$ type of dimension $2n$, which is obtained by plumbing $T^*S^n$ to the Milnor fiber of $A_k$ type, is diffeomorphic to the end connected sum of the Milnor fiber of $A_k$ type and $T^*S^n$. 
	Meanwhile, they are different as Weinstein manifolds.
\end{theorem} 
We note that the Weinstein manifolds given in Theorem \ref{thm rough statement} will be denoted by $\X$ and $\Y$ in Sections \ref{section construction} -- \ref{section diffeo}.

After proving Theorem \ref{thm rough statement}, we generalize it in two different ways. 
The first generalization is to consider the end connected sums of multiple Milnor fibers of $A$ type. 

\begin{theorem}[Technical statement is Theorem \ref{thm exotic families of Z}]
	\mbox{}
	\begin{enumerate}
		\item If $n = 3$, then 
		\begin{gather*}
			\{A_{2i_1 + \cdots + 2i_{k-1} + i_k}, A_{2i_1} \#_e A_{2i_2 + \cdots + 2i_{k-1} + i_k}, A_{2i_1} \#_e A_{2i_2} \#_e A_{2i_3 + \cdots + 2i_{k-1} + i_k},\\ 
			\cdots, A_{2i_1} \#_e A_{2i_2} \#_e \cdots \#_e A_{2i_{k-1}} \#_e A_{i_k} \} 
		\end{gather*}
		is an exotic family of different Weinstein manifolds where $A_k$ is the $2n$ dimensional Milnor fiber of $A_k$ type, and where $\#_e$ means the end connected sum.
		\item If $n \geq 5$ is odd, then 
		\begin{gather*}
			\{A_{4i_1 + \cdots + 4i_{k-1} + i_k}, A_{4i_1} \#_e A_{4i_2 + \cdots + 4i_{k-1} + i_k}, A_{4i_1} \#_e A_{4i_2} \#_e A_{4i_3 + \cdots + 4i_{k-1} + i_k},\\
			\cdots, A_{4i_1} \#_e A_{4i_2} \#_e \cdots \#_e A_{4i_{k-1}} \#_e A_{i_k} \} 
		\end{gather*}
		is an exotic family of different Weinstein manifolds.
	\end{enumerate}
\end{theorem}

The second generalization is obtained by comparing Milnor fibers of different simple singularities. 
The detailed results are given in Theorems \ref{thm exotic families} and \ref{thm exotic families 2}.

\begin{theorem}
	\label{thm exotic families}
	Let $n =2$. 
	Then, the following families of $(2n+2)=6$ dimensional manifolds have a same diffeomorphism type, but they are pairwise different as Weinstein manifolds. 
	\begin{itemize}
		\item The Milnor fibers of $A_6$ and $E_6$-singularities.
		\item The Weinstein manifold $Q_7^{2n+2}$ and the Milnor fibers of $A_7$, $E_7$, and $D_7$-singularities.
		\item The Milnor fibers of $A_8$ and $E_8$-singularities.
		\item For any $m \geq 3$, the Minor fiber of $D_{m+1}$-singularity and the Weinstein manifold $Q_m^{2n+2}$. 
		\item For $k \geq 2$, the Milnor fibers of $A_{2k+1}, D_{2k+1}$-singularities and the Weinstein manifold $Q_{2k+1}^{2n+2}$.
	\end{itemize} 
\end{theorem}
In Section \ref{section exotic familes of plumbings with names}, $Q_m^{2n}$ will be defined as an end connected sum of the Milnor fiber of $A_m$ singularity and $T^*S^n$. 

\begin{theorem}
	\label{thm exotic families 2}
	Let $n \geq 4$ be an even number. 
	Then, the following pairs of $(2n+2)$ dimensional manifolds have a same diffeomorphic type, but they are pairwise different as Weinstein manifolds. 
	\begin{itemize}
		\item The Milnor fiber of $E_7$-singularity and $Q_7^{2n+2}$.
		\item The Milnor fibers of $A_8$ and $E_8$-singularities.  
		\item The Milnor fiber of $A_{4k+1}$-singularity and $Q_{4k+1}^{2n+2}$.
		\item The Milnor fiber of $D_{4k+2}$-singularity and $Q_{4k+2}^{2n+2}$. 
		\item The Milnor fibers of $A_{4k+3}$ and $D_{4k+3}$-singularities.
	\end{itemize}
\end{theorem}

The decomposability of wrapped Fukaya categories cannot prove Theorems \ref{thm exotic families} and \ref{thm exotic families 2} completely.
Instead, we compare the symplectic cohomologies of the given Weinstein manifolds. 

\begin{remark}
	\label{rmk more generalization}
	We would like to point out that one can generalize the results of this paper easily. 
	For example, by using the same technique, one can prove the following:
	Let $W$ be a Weinstein manifold obtained by plumbing multiple $T^*S^n$ along a tree $T$ such that 
	\begin{itemize}
		\item $n \geq 3$ is odd, and 
		\item $T$ has the Dynkin tree of $A_5$ type as a subtree. 
	\end{itemize}
	Then, there are two Weinstein manifolds $W_1$ and $W_2$ such that
	\begin{itemize}
		\item the end connected sum $W_1 \#_e W_2$ is exotic to $W$,
		\item one can break the tree $T$ into two sub trees $T_1$ and $T_2$ so that $W_i$ is a plumbing of $T^*S^n$ along a tree $T_i$ for $i =1, 2$. 
	\end{itemize}
	
	We omit these generalizations for the sake of conciseness. 		
\end{remark}

This paper consists of seven sections.
The first two sections are the introduction and the preliminaries. 
In Sections \ref{section construction} -- \ref{section diffeo}, we prove Theorem \ref{thm rough statement}. 
Section \ref{section a simple extension} (resp.\ \ref{section exotic familes of plumbings with names}) contains the first (resp.\ second) generalization. 

\subsection{Acknowledgment}
\label{subsect acknoledgment} 
The authors appreciate Youngjin Bae and Hanwool Bae for the helpful discussions. 
We also appreciate Yanki Lekili for his comments on the draft of the present paper.

The first (resp.\ last) named author has been supported by a KIAS individual grant (MG079401) (resp.\ the Institute for Basic Science (IBS-R003-D1)).

\section{Preliminaries}
\label{section preliminaries}
Sections \ref{subsect Lefschetz fibration} and \ref{subsection equivalent abstract Lefschetz fibrrations} contain the definition of Lefschetz fibrations, the definition of abstract Lefschetz fibrations, the notion of stabilization, and the notion of Hurwitz moves.

We note that in the paper, we give only a brief explanation. However, in the literature, there are lots of references. 
Some of the references are \cite{Giroux-Pardon}, \cite[Section 2]{Courte}, \cite[Section 8]{Bourgeois-Ekholm-Eliashberg}, and \cite{Seidel-bluebook}. 
We also refer the reader to \cite[Section 3.1]{Casals-Murphy}, especially for the notion of stabilization of Lefschetz fibrations. 
  
In the last subsection of Section \ref{section preliminaries}, the definition of {\em end connected sum} is given. 

\subsection{Lefschetz fibration}
\label{subsect Lefschetz fibration}
Our main tool in the current paper is the notion of Lefschetz fibration, which is defined as follows:

\begin{definition}
	\label{def Lefschetz fibration}
	Let $(E,\omega = d \lambda)$ be a finite type Liouville manifold.
	See \cite{Cieliebak-Eliashberg} for the definition of Liouville manifold (of finite type). 
	A {\em Lefschetz fibration} on $E$ is a map $\pi : E \to \mathbb{C}$ satisfying the following properties:
	\begin{itemize}
		\item ({\em Lefschetz type critical points.})
		There are only finitely many points where $d\pi$ is not surjective, and for any such critical point $p$, there exist complex Darboux coordinates $(z_1, \cdots, z_n)$ centered at $p$ so that $\pi(z_1, \cdots, z_n) = \pi(p) + z_1^2 + \cdots + z_n^2$. 
		Moreover, there is at most one critical point in each fiber of $\pi$.
		\item ({\em Symplectic fiber.})
		Away from the critical points, $\omega$ is non-degenerate on the fibers of $\pi$.
		\item ({\em Triviality near the horizontal boundary.}) 
		There exists a contact manifold $(B,\xi)$, an open set $U \subset E$ such that $\pi : E \setminus U \to \mathbb{C}$ is proper and a codimension zero embedding $\Phi : U\to S_{\xi}B \times \mathbb{C}$ such that $pr_2 \circ \Phi = \pi$ and $\Phi_* \lambda = pr_1^* \lambda_\xi + pr_2^*\mu$ where $S_{\xi}B$ is the symplectization, $\mu = \frac{1}{2}r^2 d\theta$, and $(r, \theta)$ means the polar coordinates of $\mathbb{C}$.
		\item ({\em Transversality to the vertical boundary.})
		There exists $R >0$ such that the Liouville vector filed $X$ lifts the vector field $\frac{1}{2}r \partial_r$ near the region $\{|\pi| \geq R\}$.
	\end{itemize}
\end{definition}
We note that it would be more precise to use the term `Liouville Lefschetz fibration' in Definition \ref{def Lefschetz fibration} because there are Lefschetz fibrations of other types.
However, in this paper, this is the only type which we considered. 
Thus, we omit the adjective for convenience.

Let $\pi : E \to \mathbb{C}$ be a Lefschetz fibration defined on a Weinstein manifold $E$. 
Then, it is well-known that $\pi$ induces a decomposition of $E$ into two parts, one is a subcritical part, and the other is a collection of Weinstein critical handles. 
See \cite[Lemma 16.9]{Seidel-bluebook} or \cite[Section 8]{Bourgeois-Ekholm-Eliashberg}.
We note that the subcritical part is given as a product of the regular fiber of $\pi$ and $\mathbb{C}$.
In order to attaching critical handles to the subcritical part, one needs a collection of Legendrian attaching spheres which one can obtain from 
\begin{itemize}
	\item the cyclic order of the critical values of $\pi$, and
	\item the vanishing cycles corresponding to the critical values of $\pi$. 
\end{itemize}

Conversely, if one has such decomposition data of $E$, then, one can construct a Lefschetz fibration $\pi$ defined on $E$. 
We give a brief explanation after Definition \ref{def abstract Lefschetz fibration}. 

Based on the converse direction, one has an alternative definition for Definition \ref{def Lefschetz fibration}.

\begin{definition}
	\label{def abstract Lefschetz fibration}
	An {\em abstract (Weinstein) Lefschetz fibration} is a tuple 
	\[E = (F; L_1, \cdots, L_m)\]
	consisting of a Weinstein domain $F$ (the ``{\em central fiber}'') along with a finite sequence of exact parameterized Lagrangian spheres $L_1, \cdots, L_m \subset F$ (the ``{\em vanishing cycles}'').
\end{definition}

Definitions \ref{def Lefschetz fibration} and \ref{def abstract Lefschetz fibration} are interchangeable. 
In the rest of Section \ref{subsect Lefschetz fibration}, we explain the reason for that. 
For more details, we refer the reader to \cite{Seidel-bluebook} and \cite[Section 8]{Bourgeois-Ekholm-Eliashberg}. 

Let $E = (F; L_1, \cdots, L_m)$ be a given abstract Weinstein Lefschetz fibration.
Then, one can construct a Weinstein domain as follows:
First, we consider the product of $F$ and $\mathbb{D}^2$.
We remark that the product is not a Weinstein domain, because the product is a manifold with corner. 
To be more precise, we should consider a Weinstein manifold which is the product of symplectic completions of $F$ and $\mathbb{D}^2$. 
Then, we should consider a corresponding Weinstein domain, i.e., a Weinstein domain whose symplectic completion is the product Weinstein manifold. 
However, for convenience, we consider $F \times \mathbb{D}^2$. 

The vertical boundary $F \times \partial \mathbb{D}^2$ has a natural contact structure. 
Moreover, the vanishing cycle $L_i$ can be lifted to a Legendrian $\Lambda_i$ near $2 \pi i/m \in S^1$. 
We do not give the lifting procedure explicitly, but it is easily achieved by using the product structure on $F \times \mathbb{D}^2$.
We note that by assuming that the disk $\mathbb{D}^2$ has a sufficiently large radius, one could assume that the projection images of $\Lambda_i$ onto the $\partial \mathbb{D}^2 = S^1$ factor are disjoint to each other.
 
Finally, one could attach critical Weinstein handles along $\Lambda_i$ for all $i = 1, \cdots, m$.
Then, the completion of the resulting Weinstein domain admits a Lefschetz fibration satisfying that the regular fiber is $F$, and that there are exactly $m$ singular values.

\subsection{Equivalent abstract Lefschetz fibrations}
\label{subsection equivalent abstract Lefschetz fibrrations}
By \cite{Giroux-Pardon}, it is known that every Weinstein manifold admits a Lefschetz fibration. 
Also, it is known that, for a Weinstein manifold, there are infinitely many different Lefschetz fibrations. 
Then, it would be natural to ask that {\em if one has two different Lefschetz fibrations of the same total space, is there any relation between these two Lefschetz fibrations?} 
The above questions is partially answered as follows. 

Let 
\[(F; L_1, \cdots, L_m)\]
be a given abstract Lefschetz fibration. 
It is known that there are four moves producing another abstract Lefschetz fibration from $(F; L_1, \cdots, L_m)$ so that the total spaces of two abstract Lefschetz fibrations are Weinstein homotopic.
We list the four moves here. 

\begin{itemize}
	\item {\em Deformation} means a simultaneous Weinstein deformation of $F$ and exact Lagrangian isotopy of $(L_1, \cdots, L_m)$. 
	\item {\em Cyclic permutation} is to replace the ordered collection $(L_1, \cdots, L_m)$ with $(L_2, \cdots, L_m, L_1)$.
	In other words, 
	\[(F; L_1, \cdots, L_m) \simeq (F;L_2,\cdots, L_m, L_1).\] 
	The equivalence means that their total spaces are equivalent.
	\item {\em Hurwitz moves.} Let $\tau_i$ denote the symplectic Dehn twist around $L_i$. {\em Hurwitz move} is to replace $(L_1, \cdots, L_m)$ with either $(L_2, \tau_2(L_1), L_3, \cdots, L_n)$ or $(\tau_1^{-1}(L_2), L_1, L_3, \cdots, L_m)$, i.e.,
	\[(F; L_1, \cdots, L_m) \simeq (F;L_2, \tau_2(L_1), \cdots, L_m) \simeq (F; \tau_1^{-1}(L_2), L_1, \cdots, L_m).\]
	\item {\em Stabilization.} Let $\operatorname{dim}F = 2n-2$, or equivalently, the total space is of dimension $2n$. 
	For a parameterized Lagrangian disk $D^{n-1} \hookrightarrow F$ with Legendrian boundary $S^{n-2} = \partial D^{n-1} \hookrightarrow \partial F$ such that $0 = [\lambda] \in H^1(D^{n-1},\partial D^{n-1})$ where $\lambda$ is the Liouville $1$-form, replace $F$ with $\tilde{F}$, obtained by attaching a $(2n-2)$ dimensional Weinstein $(n-1)$-handle to $F$ along $\partial D^{n-1}$, and replace $(L_1, \cdots, L_m)$ with $(\tilde{L}, L_1, \cdots, L_m)$, where $\tilde{L} \subset \tilde{F}$ is obtained by gluing together $D^{n-1}$ and the core of the handle.
	In other words,
	\[(F;L_1, \cdots, L_m) \simeq (\tilde{F}; \tilde{L}, L_1, \cdots, L_m).\]
\end{itemize} 

\begin{remark}
	\label{rmk oepn quetsion}
	It is natural to ask that the above four moves are enough to connect any two Lefschetz fibrations of the same total space. 
	As far as we know, this question is still open. 
\end{remark}

\subsection{End connected sum}
\label{subsection end connected sum}
The goal of Section \ref{subsection end connected sum} is to define the notion of {\em end connected sum}.
Since the notion of end connected sum is defined as an attachment of index $1$ Weinstein handle, we start the current subsection by reviewing the notion of Weinstein handle attachment. 

In \cite{Weinstein}, Weinstein explained how to attach a Weinstein handle to a Weinstein manifold. 
A rough explanation on that is the following:
In order to a Weinstein handle $H$ to a Weinstein manifold $W$, one needs 
\begin{itemize}
	\item an {\em isotropic embedding of the attaching sphere} of $H$ onto the asymptotic boundary $\partial_\infty W$, and
	\item a {\em conformal symplectic normal bundle of $\Lambda$} where $\Lambda$ is the isotropic image of the above embedding. 
\end{itemize}
We note that a conformal symplectic normal bundle of $\Lambda$ means a conformal symplectic structure on the bundle $T\Lambda^{\perp'} / T\Lambda$ where $\perp'$ means the ``symplectic orthogonal operation'' in the tangent bundle of $W$. 
We refer the reader to \cite{Weinstein} for more details.

From the above arguments, one can induce Lemma \ref{lemma attaching 1 handle}.
\begin{lemma}
	\label{lemma attaching 1 handle}
	Let $W$ be a connected Weinstein manifold of dimension $2n \geq 4$. 
	Then, the attaching of Weinstein $1$ handle to $W$ is unique, up to Weinstein homotopy.
\end{lemma}
\begin{proof}
	Let $H$ be an index $1$ Weinstein handle of dimension $2n$.
	Then, the attaching sphere of $H$ is homeomorphic to $S^0$, i.e., two points. 
	Thus, the embedding of the attaching sphere is to choose two points from the asymptotic boundary of $W$. 
	This implies that any two isotropic embedding of the attaching sphere are isotopic if $\partial_\infty W$ is connected. 
	
	Since $W$ is a Weinstein manifold, $W$ admits a Weinstein handle decomposition. 
	Thus, $W$ is given as a union of $2n$ dimensional handles whose indices are $0, \cdots, n$. 
	This implies that $W$ does not contain an index $2n-1$ handle, $\partial_\infty W$ should be connected.  
	
	To be more precise, let $\xi$ be the contact structure given on $\partial_\infty W$. 
	Also, let $\alpha$ be the contact one form on $\partial_\infty W$ induced from the Weinstein structure of $W$. 
	If the embedding of the attaching sphere is $\{p_1, p_2\}$, then a conformal symplectic bundle on the attaching sphere would be a choice of symplectic structure on $T^{\perp'}p_i / Tp_i$. 
	
	Since $p_i$ is a point, $Tp_i$ is the zero vector space. 
	Thus, $T^{\perp'}p_i / Tp_i = T^{\perp'}p_i = \xi_{p_i}$. 
	Then, $d \alpha$ defines a unique symplectic structure on $T^{\perp'}p_i / Tp_i$, it completes the proof.  
\end{proof}

Let $W_1$ and $W_2$ be connected Weinstein manifolds of the same dimension $2n \geq 4$. 
Then, it is easy to check that there is a unique way, up to Weinstein homotopic, to construct a connected Weinstein manifold by attaching a Weinstein handle of index $1$ to $W := W_1 \cup W_2$, where $\cup$ means the disjoint union. 

\begin{definition}
	\label{def end connected sum}
	Let $W_1$ and $W_2$ be the connected Weinstein manifolds of dimension $2n \geq 4$.
	Then, the {\em end connected sum} of $W_1$ and $W_2$ is the connected Weinstein manifold obtained by attaching an index $1$ Weinstein handle to $W := W_1 \cup W_2$.  
	Let $W_1 \#_e W_2$ denote the end connected sum of $W_1$ and $W_2$. 
\end{definition} 

\section{Construction of an exotic pair ($\X, \Y$)}
\label{section construction}

\subsection{Construction}
\label{subsection construction}
As mentioned in Section \ref{section preliminaries}, from a Weinstein manifold $F$ and a cyclically ordered finite collection of exact Lagrangian spheres in $F$, one could construct another Weinstein manifold equipped with a Lefschetz fibration $\pi$ such that 
\begin{itemize}
	\item the regular fiber of $\pi$ is $F$, and
	\item the number of critical values of $\pi$ is the same as the number of exact Lagrangians in the cyclically ordered collection.
\end{itemize}
We will construct $\X$ and $\Y$ for $n \geq 2, k \geq 1$ by using the above method.  

First, we define a notation. 
\begin{definition}
	\label{def A plumbig}
	Let $A_k^{2n}$ be a $A_k$ type plumbing of $T^*S^n$, i.e., plumbings of $k$ copies of $T^*S^n$, whose plumbing graph is the Dynkin diagram of $A_k$ type. 
\end{definition}

It is well-known that $A_k^{2n}$ admits a Lefschetz fibration $\rho$ such that 
\begin{itemize}
	\item the regular fiber of $\rho$ is $T^*S^{n-1}$, 
	\item $\rho$ has $k+1$ singular values, and
	\item the corresponding vanishing cycles for singular values are the zero section of the regular fiber $T^*S^{n-1}$. 
\end{itemize} 
The existence, and the properties of $\rho$ are well-known in the literature.
We refer the reader to \cite{Khovanov-Seidel} and \cite{Wu} for more details.
 
\begin{remark}
	\label{rmk existence of rho}
	We note that if $n\geq 2$, then $\rho$ exists. 
	This is the reason why we consider the case of $n \geq 2$.  
\end{remark}

Let $F = A_2^{2n}$.
Since $F$ is a plumbing of two $T^*S^n$, one can find two exact Lagrangian spheres which are the zero sections of two $T^*S^n$.
Let those Lagrangian spheres be denoted by $\alpha$ and $\beta$. 
Then, $\rho(\alpha)$ and $\rho(\beta)$ are given by curves connecting two singular values of $\rho$.
Figure \ref{figure base_of_rho} is the picture of the base of $\rho$. 

\begin{figure}[h]
	\centering
\begingroup%
  \makeatletter%
  \providecommand\color[2][]{%
    \errmessage{(Inkscape) Color is used for the text in Inkscape, but the package 'color.sty' is not loaded}%
    \renewcommand\color[2][]{}%
  }%
  \providecommand\transparent[1]{%
    \errmessage{(Inkscape) Transparency is used (non-zero) for the text in Inkscape, but the package 'transparent.sty' is not loaded}%
    \renewcommand\transparent[1]{}%
  }%
  \providecommand\rotatebox[2]{#2}%
  \newcommand*\fsize{\dimexpr\f@size pt\relax}%
  \newcommand*\lineheight[1]{\fontsize{\fsize}{#1\fsize}\selectfont}%
  \ifx\svgwidth\undefined%
    \setlength{\unitlength}{170.07874016bp}%
    \ifx\svgscale\undefined%
      \relax%
    \else%
      \setlength{\unitlength}{\unitlength * \real{\svgscale}}%
    \fi%
  \else%
    \setlength{\unitlength}{\svgwidth}%
  \fi%
  \global\let\svgwidth\undefined%
  \global\let\svgscale\undefined%
  \makeatother%
  \begin{picture}(1,1)%
    \lineheight{1}%
    \setlength\tabcolsep{0pt}%
    \put(0,0){\includegraphics[width=\unitlength,page=1]{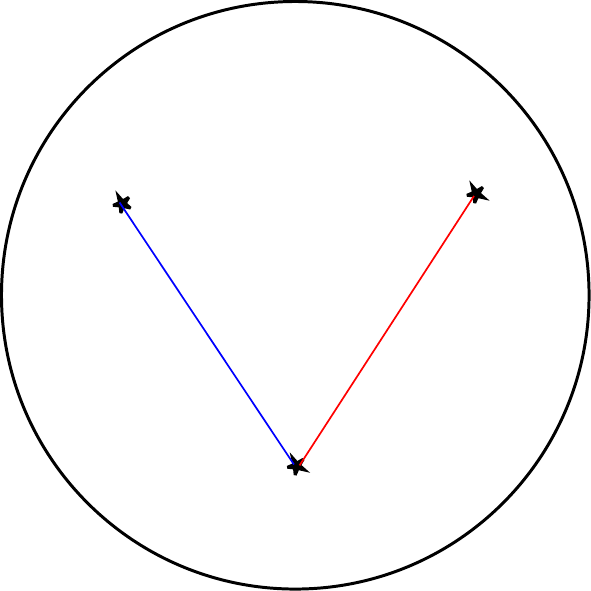}}%
    \put(0.22377545,0.40418875){\makebox(0,0)[lt]{\lineheight{1.25}\smash{\begin{tabular}[t]{l}$\rho(\alpha)$\end{tabular}}}}%
    \put(0.68282248,0.40419123){\makebox(0,0)[lt]{\lineheight{1.25}\smash{\begin{tabular}[t]{l}$\rho(\beta)$\end{tabular}}}}%
  \end{picture}%
\endgroup%
		
	\caption{Three star marked points are critical values of $\rho$, the blue (resp.\ red) curve is the image of $\alpha$ (resp.\ $\beta$) under $\rho$.}
	\label{figure base_of_rho}
\end{figure}

Since $\alpha$ and $\beta$ are Lagrangian spheres in $F$, there are generalized Dehn twists $\tau_\alpha$ and $\tau_\beta$ along $\alpha$ and $\beta$.

With above arguments, one can define $\X$ and $\Y$ as follows:
\begin{definition}
	\label{def X and Y}
	\mbox{}
	\begin{enumerate}
		\item Let $\X$ be a Weinstein manifold which is the total space of the abstract Lefschetz fibration 
		\begin{align}
			\label{eqn X}
			\left(F = A^{2n}_2; \alpha, \cdots, \alpha, (\tau_\alpha)^{2k}(\beta), \beta \right).
		\end{align}
		The total number of $\alpha$ in Equation \eqref{eqn X} is $(2k+1)$, and the total number of exact Lagrangians in the cyclically ordered collection is $(2k+3)$.
		\item Let $\Y$ be a Weinstein manifold which is the total space of the abstract Lefschetz fibration 
		\begin{align}
			\label{eqn Y}
			\left(F = A^{2n}_2; \alpha, \cdots, \alpha, \beta, \beta \right).
		\end{align}
		The total number of $\alpha$ in Equation \eqref{eqn Y} is $(2k+1)$, and the total number of exact Lagrangians in the cyclically ordered collection is $(2k+3)$.
	\end{enumerate}
\end{definition}

\subsection{Wrapped Fukaya category of $\X$}
\label{subsection wrapped Fukaya of X}
In Section \ref{subsection construction}, we constructed $\X$ and $\Y$ as total spaces of abstract Lefschetz fibrations. 
Before going further, we investigate the wrapped Fukaya categories of $\X$ and $\Y$ in Sections \ref{subsection wrapped Fukaya of X} and \ref{subsection wrapped Fukaya of Y}.
By studying those wrapped Fukaya categories, we can have a hint for proving that $\X$ and $\Y$ are not equivalent as Weinstein manifolds. 

We start with Lemma \ref{lemma strucutre of X}.
\begin{lemma}
	\label{lemma strucutre of X}
	The Weinstein manifold $\X$ defined in Definition \ref{def X and Y} is Weinstein homotopic to $A^{2n+2}_{2k+1}$.
\end{lemma}
\begin{proof}
	We prove Lemma \ref{lemma strucutre of X} by operating a sequence of moves given in Section \ref{subsection equivalent abstract Lefschetz fibrrations}.
	
	First, we can operate Hurwitz moves $(2k+1)$ times, which are moving $(\tau_\alpha)^{2k}(\beta)$ to left in Equation \eqref{eqn X}. 
	Since, on the left side of $(\tau_\alpha)^{2k}(\beta)$, there are $(2k+1)$ many $\alpha$, we obtain the following:
	\[\X \simeq \left(A^{2n}_2; \alpha, \cdots, \alpha, (\tau_\alpha)^{2k}(\beta), \beta \right) \simeq \left(A^{2n}_2; \tau_\alpha^{-1}(\beta), \alpha, \alpha, \cdots, \alpha, \beta \right).\]
	
	Second, we can move the first $\tau_{\alpha}^{-1}(\beta)$ to the right end of the collection of exact Lagrangians.
	This is because the collection is cyclically ordered. 
	Then, we have 
	\[\X \simeq \left(A^{2n}_2; \alpha, \cdots, \alpha, \beta, \tau_\alpha^{-1}(\beta)  \right).\]
	
	After that, we move $\tau_\alpha^{-1}(\beta)$ to the left once.
	It concludes that 
	\[\X \simeq \left(A^{2n}_2; \alpha, \cdots, \alpha, (\tau_\beta^{-1} \circ \tau_\alpha^{-1})(\beta), \beta \right).\]
	
	By using the property of $\rho : A^{2n}_{2} \to \mathbb{C}$, one can easily check that 
	\[(\tau_\beta^{-1} \circ \tau_\alpha^{-1})(\beta) \simeq \alpha.\]
	Thus, 
	\[\X \simeq \left(A^{2n}_2; \alpha, \cdots, \alpha, \alpha, \beta \right).\]
	We note that the number of $\alpha$ in the above abstract Lefschetz fibration is $(2k+2)$. 
	
	From the definition of the operation `stabilization', it is easy to show that the right side of the above equation is obtained by stabilizing the following abstract Lefschetz fibration  
	\[\left(A^{2n}_1 \simeq T^*S^n; \alpha, \cdots, \alpha \right),\] 
	which is the well-known abstract Lefschetz fibration of $A^{2n+2}_{2k+1}$. 
	This proves that $\X \simeq A^{2n+2}_{2k+1}$. 
\end{proof}

Lemma \ref{lemma wrapped Fukaya of X} follows naturally. 
\begin{lemma}
	\label{lemma wrapped Fukaya of X}
	We have 
	\[\cW(\X) \simeq \cW(A^{2n+2}_{2k+1}),\]
	where $\cW(W)$ means the wrapped Fukaya category of $W$.
\end{lemma}

\begin{remark}
	For the definition of wrapped Fukaya categories we used in the present paper, see \cite{Ganatra-Pardon-Shende1}.
	Roughly, wrapped Fukaya categories can be defined as $A_\infty$-categories whose objects are exact Lagrangians with cylindrical ends.
	In order to be precise, we note that the notion of equivalence between wrapped Fukaya categories in this paper is the pretriangulated equivalence, i.e., quasi-equivalence after taking pretriangulated closures. 
	In the rest of the paper, we simply say `equivalence' rather than `pretriangulated equivalence'.
\end{remark}

\subsection{Wrapped Fukaya category of $\Y$}
\label{subsection wrapped Fukaya of Y}
In the next subsection, we prove the following Lemma.
\begin{lemma}
	\label{lemma wrapped Fukaya of Y}
	The wrapped Fukaya category of $\Y$ is equivalent to the coproduct of the wrapped Fukaya categories of $A^{2n+2}_{2k}$ and $T^*S^{n+1}$, i.e.,
	\[\cW(\Y) \simeq \cW(A^{2n+2}_{2k}) \amalg \cW(T^*S^{n+1}).\]
\end{lemma}

Because of the length of the proof, we give a sketch of the proof in the present subsection. 

In order to prove Lemma \ref{lemma wrapped Fukaya of Y}, we construct a Weinstein sectorial covering $\{\overline{Y}_1, \overline{Y}_2\}$ of a Weinstein domain which is Weinstein homotopic to $\Y$.
In the construction of the Weinstein sectorial covering, we need to use the notions of Weinstein homotopy, symplectic completion, Lagrangian skeleta, etc. 
The construction would take the most part of Section \ref{subsection wrapped Fukaya of Y}. 

The constructed Weinstein sectorial covering $\{\overline{Y}_1, \overline{Y}_2\}$ satisfies that $\overline{Y}_1 \simeq A^{2n+2}_{2k}$ and $\overline{Y}_2 \simeq T^*S^{n+1}$.
The notion of equivalences here is Weinstein homotopic. 
Also, after a proper modification of the Liouville structures, the intersection $\overline{Y}_1 \cap \overline{Y}_2$ is equivalent to $T^*[0,1] \times \mathbb{D}^{2n}$ whose wrapped Fukaya category vanishes. 

By the result of \cite{Ganatra-Pardon-Shende2}, we conclude that 
\[\cW(\Y) \simeq \hocolim \left( \begin{tikzcd}[ampersand replacement=\&]
	\cW(\overline{Y}_1) \& \& \cW(\overline{Y}_2)  \\
	\& \cW(\overline{Y}_1 \cap \overline{Y}_2) \arrow{lu} \arrow{ru} 
\end{tikzcd}\right).\]
Then, the above equation proves Lemma \ref{lemma wrapped Fukaya of Y}.

\subsection{Proof of Lemma \ref{lemma wrapped Fukaya of Y}}
\label{subsection proof of lemma}
We note that in the current subsection, we modify Weinstein structures via Weinstein homotopies. 
For the notion of Weinstein homotopies, see \cite{Cieliebak-Eliashberg}.
Also, we refer the reader to \cite{Starkston}.

Before studying $\cW(\Y)$, we recall the construction of $\Y$. 
By Equation \eqref{eqn Y}, $\Y$ is obtained by attaching $(2k+3)$ Weinstein critical handles to the subcritical part. 
Let $H_1, \cdots, H_{2k+3}$ denote the critical handles which are attached to the subcritical part of $\Y$ such that 
\begin{itemize}
	\item $H_1, \cdots, H_{2k+1}$ are attached along the Legendrians induced from $\alpha$, and
	\item $H_{2k+2}$ and $H_{2k+3}$ are attached along the Legendrians induced from $\beta$.	
\end{itemize}

We discuss a decomposition of the subcritical part of $\Y$. 
Since the subcritical part is given as a product of the fiber $A_2^{2n}$ and a disk $\mathbb{D}^2$, we will decompose $A_2^{2n}$ and $\mathbb{D}^2$.

\begin{remark}
	\label{rmk omitting up to completion}
	To be more precise, we should say that the subcritical part of $\Y$ is a Weinstein domain whose symplectic completion is $A^{2n}_2 \times \mathbb{C}$. 
	Then, $\Y$ is the symplectic completion of a Weinstein domain which is obtained by attaching $H_1, \cdots, H_{2k+3}$ to the subcritical part. 
	However, for convenience, we omit the words `up to symplectic completion' in the current paper. 
\end{remark}

In order to decompose $A_2^{2n}$, we recall that $A_2^{2n}$ is a union of $T^*\alpha$ and $T^*\beta$, as mentioned in Section \ref{subsection construction}. 
This induces that $A_2^{2n}$ admits a Weinstein handle decomposition which consists of one $2n$-dimensional index $0$-handle and two $2n$-dimensional index $n$-handles. 
Moreover, the union of the unique $0$-handle and one $n$-handle is $T^*\alpha$, and the union of the $0$-handle and the other $n$-handle is $T^*\beta$.
We use $h_0, h_\alpha, h_\beta$ to denote the index $0$-handle, the index $n$-handle making $T^*\alpha$, the index $n$-handle making $T^*\beta$ of the fiber $A_2^{2n}$, respectively. 

For a decomposition of $\mathbb{D}^2$, we consider a Weinstein sectorial covering of $\mathbb{D}^2$.
In order to do this, we start from a Weinstein handle decomposition of $\mathbb{D}^2$, consisting of three index $0$-handles and two index $1$-handles. 
Figure \ref{figure handle_decomposition_of_D2}, $a)$ describes the handle decomposition. 

\begin{figure}[h]
	\centering
\begingroup%
  \makeatletter%
  \providecommand\color[2][]{%
    \errmessage{(Inkscape) Color is used for the text in Inkscape, but the package 'color.sty' is not loaded}%
    \renewcommand\color[2][]{}%
  }%
  \providecommand\transparent[1]{%
    \errmessage{(Inkscape) Transparency is used (non-zero) for the text in Inkscape, but the package 'transparent.sty' is not loaded}%
    \renewcommand\transparent[1]{}%
  }%
  \providecommand\rotatebox[2]{#2}%
  \newcommand*\fsize{\dimexpr\f@size pt\relax}%
  \newcommand*\lineheight[1]{\fontsize{\fsize}{#1\fsize}\selectfont}%
  \ifx\svgwidth\undefined%
    \setlength{\unitlength}{283.46456693bp}%
    \ifx\svgscale\undefined%
      \relax%
    \else%
      \setlength{\unitlength}{\unitlength * \real{\svgscale}}%
    \fi%
  \else%
    \setlength{\unitlength}{\svgwidth}%
  \fi%
  \global\let\svgwidth\undefined%
  \global\let\svgscale\undefined%
  \makeatother%
  \begin{picture}(1,0.55)%
    \lineheight{1}%
    \setlength\tabcolsep{0pt}%
    \put(0,0){\includegraphics[width=\unitlength,page=1]{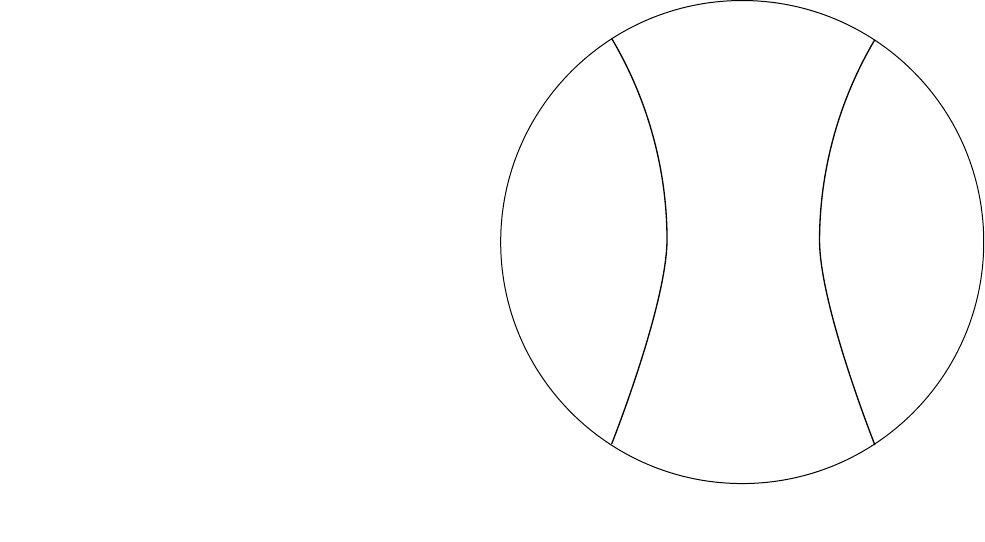}}%
    \put(0.56038708,0.3723567){\makebox(0,0)[lt]{\lineheight{1.25}\smash{\begin{tabular}[t]{l}$S_1$\end{tabular}}}}%
    \put(0.88395849,0.37235783){\makebox(0,0)[lt]{\lineheight{1.25}\smash{\begin{tabular}[t]{l}$S_2$\end{tabular}}}}%
    \put(0.73879445,0.20781972){\makebox(0,0)[lt]{\lineheight{1.25}\smash{\begin{tabular}[t]{l}$T$\end{tabular}}}}%
    \put(0,0){\includegraphics[width=\unitlength,page=2]{handle_decomposition_of_D2.pdf}}%
    \put(0.22028349,0.02031746){\makebox(0,0)[lt]{\lineheight{1.25}\smash{\begin{tabular}[t]{l}$a)$\end{tabular}}}}%
    \put(0.73059222,0.02031268){\makebox(0,0)[lt]{\lineheight{1.25}\smash{\begin{tabular}[t]{l}$b)$\end{tabular}}}}%
  \end{picture}%
\endgroup%
		
	\caption{$a)$ describes the handle decomposition of $\mathbb{D}^2$, consisting of three $0$ handles and two $1$ handles. The dashed lines are boundaries of handles, the dots are centers of handles, and the arrows describe the Liouville vector field. 
	$b)$ describes the decomposition of $A_1 \cup A_2 \cup B = \mathbb{D}^2 $. The part of $\partial B$, which is located inside the circle, is given by a union of unstable manifolds of the centers of two $1$ handles, with respect to the Liouville vector flow.}
	\label{figure handle_decomposition_of_D2}
\end{figure}

From the Weinstein handle decomposition, the centers of two $1$-handles have curves as their unstable manifolds with respect to the Liouville flow.
Then, those two curves divide $\mathbb{D}^2$ into three parts. 
This decomposition of $\mathbb{D}^2$ is given in Figure \ref{figure handle_decomposition_of_D2}, $b)$. 
Let $T$ be the piece of the decomposition such that the boundary of $T$ contains the both unstable manifolds.
Also, let $S_1, S_2$ denote the other two pieces. 

Based on the above arguments, the subcritical part $A_2^{2n} \times \mathbb{D}^2$ admits a decomposition into a union of nine pieces 
\[	h_0 \times S_1, h_0 \times S_2, h_0 \times T, h_\alpha \times S_1, h_\alpha \times S_2, h_\alpha \times T, h_\beta \times S_1, h_\beta \times S_2, h_\beta \times T.\]
One can observe that $h_\alpha \times (S_1 \cup S_2 \cup T)$ and $h_\beta \times (S_1 \cup S_2 \cup T)$ are equivalent to $(2n+2)$-dimensional index $n$-Weinstein handles, up to Weinstein homotopy.
In other words, we can construct the subcritical part $A_2^{2n} \times \mathbb{D}^2$ by attaching two Weinstein handles $H_\alpha$ and $H_\beta$ to $h_0 \times (S_1 \cup S_2 \cup T)$, instead of attaching $h_\alpha \times (S_1 \cup S_2 \cup T)$ and $h_\beta \times (S_1 \cup S_2 \cup T)$. 
The notation $H_\alpha$ (resp.\ $H_\beta$) denotes the Weinstein handle replacing $h_\alpha \times (S_1 \cup S_2 \cup T)$ (resp.\ $h_\beta \times (S_1 \cup S_2 \cup T)$).

We note that by attaching $H_\alpha$ and $H_\beta$ to  $h_0 \times (S_1 \cup s_2 \cup T)$, we obtain a Weinstein domain which does not admit a product Liouville structure.
In other words, by replacing $h_\alpha \times (S_1 \cup S_2 \cup T)$ (resp.\ $h_\beta \times (S_1 \cup S_2 \cup T)$) with $H_\alpha$ (resp.\ $H_\beta$), we modify the Liouville structure on them, and the modified structures do not respect the product structure.

To summarize, we decompose the subcritical part as a union of the following five pieces,
\begin{gather}
	\label{eqn pieces}
	h_0 \times S_1, h_0 \times S_2, h_0 \times T, H_\alpha, H_\beta.
\end{gather}
We note that the decomposition is not a Weinstein handle decomposition. 
This is because the first three pieces are not Weinstein handles.

From the above descriptions, we have a decomposition of $\Y$ into the union of five pieces in \eqref{eqn pieces} and $H_i$ for $1 \geq i \geq 2k+3$.
For the critical handles, without loss of generality, one can assume that the critical handles $H_1, \cdots, H_{2k+1}$ (resp.\ $H_{2k+2}, H_{2k+3}$) are attached to the $A_2^{2n} \times S_1$ (resp.\ $A_2^{2n} \times S_2$).
Figure \ref{figure base_of_Y} describes it. 
\begin{figure}[h]
	\centering
\begingroup%
  \makeatletter%
  \providecommand\color[2][]{%
    \errmessage{(Inkscape) Color is used for the text in Inkscape, but the package 'color.sty' is not loaded}%
    \renewcommand\color[2][]{}%
  }%
  \providecommand\transparent[1]{%
    \errmessage{(Inkscape) Transparency is used (non-zero) for the text in Inkscape, but the package 'transparent.sty' is not loaded}%
    \renewcommand\transparent[1]{}%
  }%
  \providecommand\rotatebox[2]{#2}%
  \newcommand*\fsize{\dimexpr\f@size pt\relax}%
  \newcommand*\lineheight[1]{\fontsize{\fsize}{#1\fsize}\selectfont}%
  \ifx\svgwidth\undefined%
    \setlength{\unitlength}{226.77165354bp}%
    \ifx\svgscale\undefined%
      \relax%
    \else%
      \setlength{\unitlength}{\unitlength * \real{\svgscale}}%
    \fi%
  \else%
    \setlength{\unitlength}{\svgwidth}%
  \fi%
  \global\let\svgwidth\undefined%
  \global\let\svgscale\undefined%
  \makeatother%
  \begin{picture}(1,1)%
    \lineheight{1}%
    \setlength\tabcolsep{0pt}%
    \put(0,0){\includegraphics[width=\unitlength,page=1]{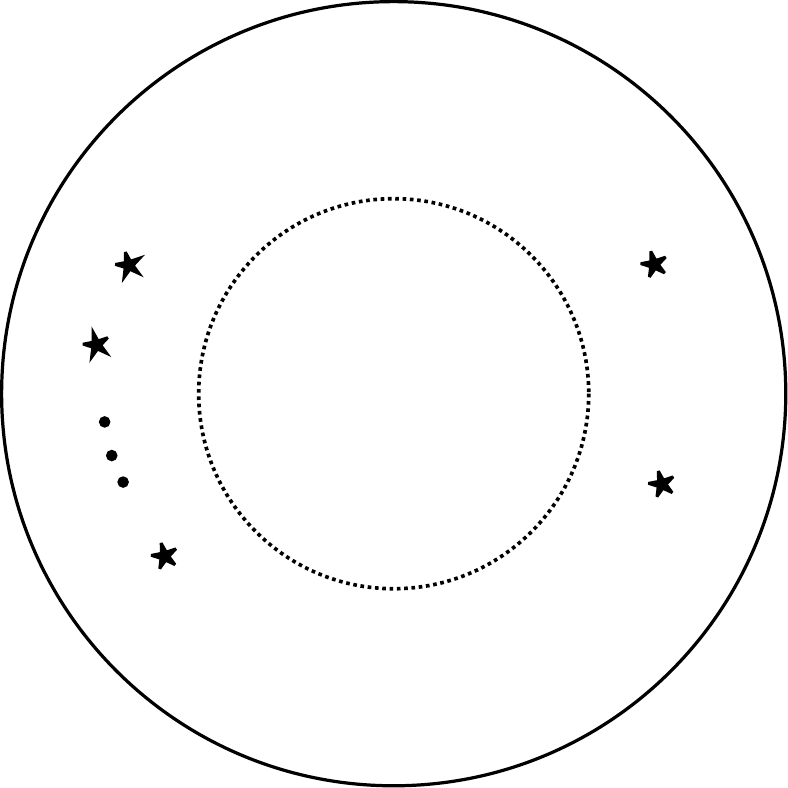}}%
    \put(0.1101611,0.61870225){\makebox(0,0)[lt]{\lineheight{1.25}\smash{\begin{tabular}[t]{l}$\alpha$\end{tabular}}}}%
\put(0.07155217,0.51903934){\makebox(0,0)[lt]{\lineheight{1.25}\smash{\begin{tabular}[t]{l}$\alpha$\end{tabular}}}}%
\put(0.16379107,0.25784074){\makebox(0,0)[lt]{\lineheight{1.25}\smash{\begin{tabular}[t]{l}$\alpha$\end{tabular}}}}%
\put(0.76804967,0.60988667){\makebox(0,0)[lt]{\lineheight{1.25}\smash{\begin{tabular}[t]{l}$\beta$\end{tabular}}}}%
\put(0.76515965,0.3352067){\makebox(0,0)[lt]{\lineheight{1.25}\smash{\begin{tabular}[t]{l}$\beta$\end{tabular}}}}%
    \put(0,0){\includegraphics[width=\unitlength,page=2]{base_of_Y.pdf}}%
    \put(0.30307267,0.56938348){\makebox(0,0)[lt]{\lineheight{1.25}\smash{\begin{tabular}[t]{l}$S_1$\end{tabular}}}}%
    \put(0.63203702,0.56938462){\makebox(0,0)[lt]{\lineheight{1.25}\smash{\begin{tabular}[t]{l}$S_2$\end{tabular}}}}%
    \put(0.48445354,0.40210423){\makebox(0,0)[lt]{\lineheight{1.25}\smash{\begin{tabular}[t]{l}$T$\end{tabular}}}}%
  \end{picture}%
\endgroup%
		
	\caption{The big circle means the base of $\pi: \Y \to \mathbb{C}$, the small dotted circle means the boundary of radius $1$ disk $\mathbb{D}_1^2$ which is the union of $S_1, S_2$, and $T$, i.e., the subcritical part is given as the inverse image of the dotted disk. The star marked points are singular values of $\pi$ decorated with the vanishing cycles, i.e., $H_1, \cdots, H_{2k+1}$ (resp.\ $H_{2k+2}, H_{2k+3}$) are attached to the `left' or $S_1$ (resp.\ `right' or $S_2$) side of the subcritical part.}
	\label{figure base_of_Y}
\end{figure}
 
For the later use, we modify the Weinstein domain which we obtained by attaching five pieces in \eqref{eqn pieces}. 
We observe that the attaching regions of $H_\alpha$ and $H_\beta$ are contained in $\partial \big(h_0 \times (S_1 \cup S_2 \cup T)\big)$. 
However, we can modify it so that the attaching region of $H_\alpha$ (resp.\ $H_\beta$) is contained in $\partial h_0 \times S_1$ (resp.\ $h_0 \times S_2$). 
We note that the modification does not change the symplectic completion of the resulting Weinstein domain, up to Weinstein homotopy. 

By the above modification, one can observe that the resulting Weinstein domain, even after attaching critical handles, does not change up to Weinstein homotopy. 
However, the Lagrangian skeleton changes. 
Figure \ref{figure skeleton} describes the change on Lagrangian skeleta.

\begin{figure}[h]
	\centering
\begingroup%
  \makeatletter%
  \providecommand\color[2][]{%
    \errmessage{(Inkscape) Color is used for the text in Inkscape, but the package 'color.sty' is not loaded}%
    \renewcommand\color[2][]{}%
  }%
  \providecommand\transparent[1]{%
    \errmessage{(Inkscape) Transparency is used (non-zero) for the text in Inkscape, but the package 'transparent.sty' is not loaded}%
    \renewcommand\transparent[1]{}%
  }%
  \providecommand\rotatebox[2]{#2}%
  \newcommand*\fsize{\dimexpr\f@size pt\relax}%
  \newcommand*\lineheight[1]{\fontsize{\fsize}{#1\fsize}\selectfont}%
  \ifx\svgwidth\undefined%
    \setlength{\unitlength}{283.46456693bp}%
    \ifx\svgscale\undefined%
      \relax%
    \else%
      \setlength{\unitlength}{\unitlength * \real{\svgscale}}%
    \fi%
  \else%
    \setlength{\unitlength}{\svgwidth}%
  \fi%
  \global\let\svgwidth\undefined%
  \global\let\svgscale\undefined%
  \makeatother%
  \begin{picture}(1,0.59)%
    \lineheight{1}%
    \setlength\tabcolsep{0pt}%
    \put(0,0){\includegraphics[width=\unitlength,page=1]{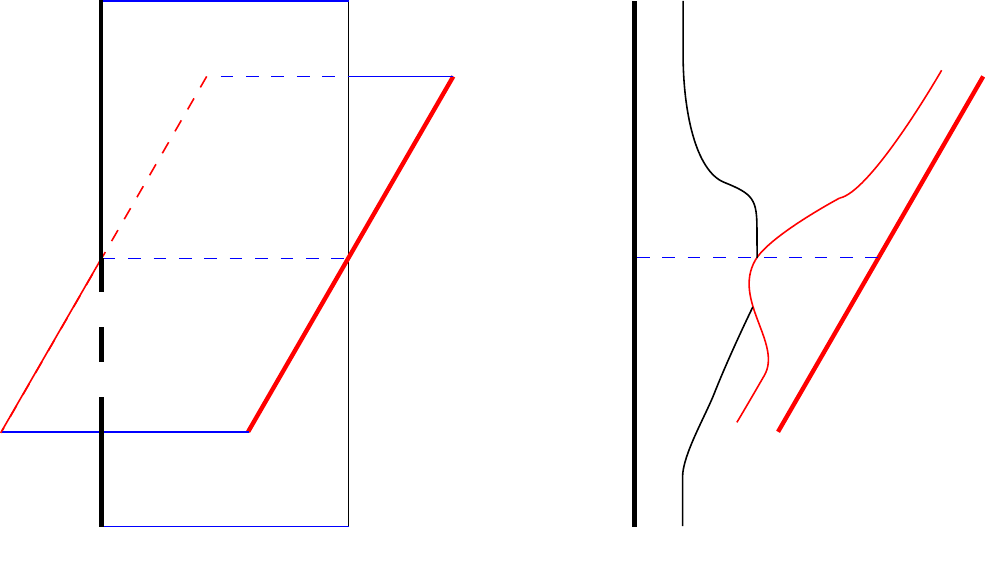}}%
    \put(0.17753204,0.01024532){\makebox(0,0)[lt]{\lineheight{1.25}\smash{\begin{tabular}[t]{l}$a).$\end{tabular}}}}%
    \put(0.74708847,0.01024269){\makebox(0,0)[lt]{\lineheight{1.25}\smash{\begin{tabular}[t]{l}$b).$\end{tabular}}}}%
    \put(0,0){\includegraphics[width=\unitlength,page=2]{skeleton.pdf}}%
  \end{picture}%
\endgroup%
		
	\caption{$a)$ is a conceptual picture describing the Lagrangian skeleton of the subcritical part $A_2^{2n} \times \mathbb{D}^2 \subset \Y$. We note that we are considering the product Weinstein structure for $a)$. The black, red, and blue lines correspond to $\alpha$ in $A_2^{2n}$, $\beta$ in $A_2^{2n}$, and the Lagrangian skeleton of $S_1 \cup S_2 \cup T = \mathbb{D}^2$, respectively. We note that the Lagrangian skeleton of $\Y$ is constructed by attaching $2k+3$ Lagrangian disks to $a)$, which correspond to the $2k+3$ critical handles. The $2k+1$ Lagrangian disks corresponding to $H_1, \cdots, H_{2k+1}$ are attached to the thick black line, and the other two Lagrangian disks corresponding to $H_{2k_2}, H_{2k+3}$ are attached to the thick red line.
	$b)$ is the Lagrangian skeleton of the subcritical part in the modified Weinstein domain. We note that there are no changes on the thick black and red lines.}
	\label{figure skeleton}
\end{figure}

Thanks to the modification, we can set 
\begin{gather}
	\label{eqn barY}
	\overline{Y} := \left(h_0 \times (S_1 \cup S_2 \cup T)\right) \cup H_\alpha \cup H_\beta \cup (\bigcup_{i=1}^{2k+3}H_i),\\
	\label{eqn bar Y_1}
	\overline{Y}_1 := h_0 \times (S_1 \cup T) \cup H_\alpha \cup \left(\bigcup_{i=1}^{2k+1} H_i\right), \\
	\label{eqn bar Y_2}
	\overline{Y}_2 := h_0 \times (S_2 \cup T) \cup H_\beta \cup \left(\bigcup_{i=2}^{3} H_{2k+i}\right).
\end{gather} 
Then, $\overline{Y}$ is equivalent to the original $\Y$ up to Weinstein homotopy. 
We note that Equations \eqref{eqn bar Y_1} and \eqref{eqn bar Y_2} make sense since the attaching region of $H_\alpha$ (resp.\ $H_\beta$) is contained in $h_0 \times S_1$ (resp.\ $h_0 \times S_2$).

This induces that
\begin{gather}
	\label{eqn equivalent Fukaya categories}
	\cW(\Y) \simeq \cW(\overline{Y}).
\end{gather}
Moreover, since $\{\overline{Y}_1, \overline{Y}_2\}$ is a Weinstein sectorial covering of $\overline{Y}$, the result of \cite{Ganatra-Pardon-Shende2} and Equation \eqref{eqn equivalent Fukaya categories} conclude that
\[\cW(\Y)\simeq\hocolim\left(\begin{tikzcd}[column sep=0.1cm]
\cW(\overline{Y}_1) & & \cW(\overline{Y}_2)\\
& \cW(\overline{Y}_1 \cap \overline{Y}_2)\ar[lu]\ar[ru]
\end{tikzcd}\right).\]
In order to complete the proof, we consider the wrapped Fukaya categories of Weinstein sectors $\overline{Y}_1, \overline{Y}_2$, and $\overline{Y}_1 \cap \overline{Y}_2$. 

For $\cW(\overline{Y}_1 \cap \overline{Y}_2)$, we observe that 
\[\overline{Y}_1 \cap \overline{Y}_2 = h_0 \times T \simeq \mathbb{D}^{2n} \times T.\]
Moreover, from the definition of $T$, one can easily check that as a Weinstein sector, $T$ is equivalent to $T^*[0,1]$. 
Thus, we have 
\[\cW(\overline{Y}_1 \cap \overline{Y}_2) \simeq \cW(\mathbb{D}^{2n} \times T^*[0,1]) \simeq \cW(\mathbb{D}^{2n}).\]
Since $\cW(\mathbb{D}^{2n})$ is quasi-equivalent to the zero category, so is the wrapped Fukaya category of the intersection.

For $\cW(\overline{Y}_1)$ (resp.\ $\cW(\overline{Y}_2)$), we consider the convex completion of the Weinstein sector $\overline{Y}_1$ (resp.\ $\overline{Y}_2$).
For the notion of convex completion, see \cite[Lemma 2.32]{Ganatra-Pardon-Shende1}.

By taking the convex completion of $\overline{Y}_1$ (resp.\ $\overline{Y}_2$), one obtains 
\begin{gather}
	\label{eqn Y_1}
	Y_1 := \big(h_0 \times (S_1 \cup S_2 \cup T)\big) \cup H_\alpha \cup \big(\bigcup_{i=1}^{2k+1}H_i\big),\\
	\label{eqn Y_2}
	Y_2 := \big(h_0 \times (S_1 \cup S_2 \cup T)\big) \cup H_\beta \cup \big(\bigcup_{i=2}^3H_{2k+i}\big),
\end{gather}
together with a stop.
Since the boundary of the Liouville sector $\overline{Y}_1$ (resp.\ $\overline{Y}_2$) is given by a product of $h_0 \simeq \mathbb{D}^{2n}$ and an $1$ dimensional curve, the stop is a Legendrian disk.
We note that the $1$ dimensional curve above is given in Figure \ref{figure handle_decomposition_of_D2}, as a part of the boundary of $S_1 \cup T$ (resp.\ $S_2 \cup T$), contained in the interior disk.

Since the stop is a disk, we have the followings.
\[\cW(\overline{Y}_1) \simeq \cW(Y_1), \text{  and  } \cW(\overline{Y}_2) \simeq \cW(Y_2).\]
Thus, we consider $\cW(Y_1)$ (resp.\ $\cW(Y_2)$) instead of $\cW(\overline{Y}_1)$ (resp.\ $\cW(\overline{Y}_2)$). 

In order to study $Y_1$ (resp.\ $Y_2$), we consider the construction of $Y_1$ (resp.\ $Y_2$).
Equation \eqref{eqn Y_1} (resp.\ Equation \eqref{eqn Y_2}) means that $Y_1$ (resp.\ $Y_2$) is obtained by attaching $H_\alpha$ and critical handles $H_1, \cdots, H_{2k+1}$ to a Weinstein domain
\[h_0 \times (S_1 \cup S_2 \cup T).\] 
Since $S_1 \cup S_2 \cup T = \mathbb{D}^2$, and since $h_0$ is a $0$-handle, it is equivalent to attach $H_\alpha$ and $H_1 \cdots, H_{2k+1}$ to the $2n+2$ dimensional index $0$-handle. 

From the above argument, it is easy to check that $Y_1$ is a Weinstein domain which is equivalent to the total space of an abstract Lefschetz fibration 
\[(T^*\alpha \simeq T^*S^n ; \alpha, \cdots, \alpha),\]
where the number of $\alpha$ above is $2k+1$. 
It is known that the total space of the abstract Lefschetz fibration is $A_{2k}^{2n+2}$. 
Thus, it concludes that 
\[\cW(\overline{Y}_1) \simeq \cW(A_{2k}^{2n+2}).\]

Similarly, $Y_2$ is equivalent to the total space of 
\[(T^*\beta \simeq T^*S^n ; \beta, \beta),\]
which is $T^*S^{n+1}$. 

It concludes that 
\[\cW(\Y)\simeq\hocolim\left(\begin{tikzcd}[column sep=0.1cm]
	\cW(A_{2k}^{2n+2}) & & \cW(T^*S^{n+1})\\
	& 0 \ar[lu]\ar[ru]
\end{tikzcd}\right),\]
i.e., 
\[\cW(\Y) \simeq \cW(A_{2k}^{2n+2}) \amalg \cW(T^*S^{n+1}).\]
\qed

\begin{remark}
	\label{rmk facts}
	We remark the following four facts before going further. 
	\begin{enumerate}
		\item From the proof of Lemma \ref{lemma wrapped Fukaya of Y}, one can observe that $\Y$ is constructed from $\overline{Y}_1 \simeq A^{2n+2}_{2k}$ and $\overline{Y}_2 \simeq T^*S^{n+1}$, by a `gluing construction' using a hypersurface. The gluing construction is described in \cite[Section 3.1]{Eliashberg}. Also, see \cite{Avdek}. 
		The hypersurface which used to glue $\overline{Y}_1$ and $\overline{Y}_2$ is a $(2n+1)$-dimensional disk, which corresponds to $\mathbb{D}^{2n} \times T^*_{\frac{1}{2}}[0,1]$ in $\mathbb{D}^{2n} \times T^*[0,1] \simeq \overline{Y}_1 \cap \overline{Y}_2$.
		
		In other words, $\Y$ is constructed using the notion of {\em end connected sum}.
		Then, one can observe that $\Y$ is obtained by taking an end connected sum of $A^{2n+2}_{2k}$ and $T^*S^{n+1}$.
		The notion of end connected sum is used in \cite{McLean}, in order to construct exotic symplectic structures on $\mathbb{R}^{2m}$ for $m \geq 4$.
		\item One can observe that both of $\X$ and $\Y$ can be constructed from $A_{2k}^{2n+2}$. 
		More precisely, since $\X$ is equivalent to $A_{2k+1}^{2n+2}$, $\X$ is obtained by plumbing $T^*S^{n+1}$ to $A_{2k}^{2n+2}$. 
		From the view point of Weinstein handle decomposition, plumbing $T^*S^{n+1}$ is equivalent to add a critical Weinstein handle. 
		Similarly, one can observe that the construction of $\X$ in Definition \ref{def X and Y} is equivalent to add a critical Weinstein handle and a canceling pair of Weinstein handles of index $n$ and $(n+1)$ to $A_{2k}^{2n+2}$. 
		
		One also can observe that the construction of $\Y$ in Definition \ref{def X and Y} is also equivalent to add a critical Weinstein handle and a canceling pair of Weinstein handles of index $n$ and $(n+1)$ to $A_{2k}^{2n+2}$, in a different way from the case of $\X$.
		The argument in Section \ref{subsection proof of lemma} explains a way to convert the canceling pair in $\Y$ as another canceling pair of Weinstein handles of index $0$ and $1$. 
		Then, the added critical handle and the index $0$ handle in the canceling pair construct $T^*S^{n+1}$, and the index $1$ handle in the canceling pair becomes the index $1$ handle `gluing' $A_{2k}^{2n+2}$ and $T^*S^{n+1}$.
		\item The above arguments give a way to construct a Lefschetz fibration of $Y$ where $Y$ is obtained by a end connected sum of Weinstein manifolds. 
		For example, if $Y$ is the end connected sum of three $T^*S^{n+1}$, then $Y$ is a total space of the abstract Lefschetz fibration 
		\[(A_3^{2n}; \alpha, \alpha, \beta, \beta, \gamma, \gamma),\]
		where $A_3^{2n}$ is the plumbing of three cotangent bundles of spheres $T^*\alpha$, $T^*\beta$, and $T^*\gamma$. 
		\item One can observe that the same technique, i.e., to use Lefschetz fibrations in order to construct Weinstein sectorial coverings, works if we calculate the wrapped Fukaya category of a total space of a Lefschetz fibration.
		However, taking the homotopy colimit for a general Lefschetz fibration would not as simple as the case of $\Y$. 
		For example, one can compute $\cW(\X)$ by using the same method. 
		However, for taking the homotopy colimit for $\X$, we should care the plumbing sector which is described in \cite[Section 6.2]{Ganatra-Pardon-Shende3}. 
		Then, the homotopy colimit formula in \cite{Karabas-Lee} will give the resulting wrapped Fukaya category. 
	\end{enumerate}
\end{remark}

\section{Different symplectic structures on $\X$ and $\Y$}
\label{section not sympelcto}
In Section \ref{section construction}, we constructed $\X$ and $\Y$. 
Also, we proved Lemmas \ref{lemma wrapped Fukaya of X} and \ref{lemma wrapped Fukaya of Y}.
By using them, we prove that $\X$ and $\Y$ have different wrapped Fukaya categories in Section \ref{section not sympelcto}. 

In order to distinguish their symplectic structures, we focus on the result of Lemma \ref{lemma wrapped Fukaya of Y}.
Lemma \ref{lemma wrapped Fukaya of Y} says that $\cW(\Y)$ can be written as a coproduct of two nontrivial categories, up to equivalence. 
On the other hand, Lemma \ref{lemma wrapped Fukaya of X} says that $\cW(\X) \simeq \cW(A^{2n+2}_{2k+1})$. 
Thus, if $\cW(A^{2n+2}_{2k+1})$ cannot be written as a coproduct in a nontrivial way, i.e., a coproduct of itself and something equivalent to the empty category. 
This completes the proof. 

In Section \ref{subsection decomposable lemme}, we prove Lemma \ref{lemma decomposable} describing a property of coproduct, and in Section \ref{subsection different wrapped Fukaya categories}, we show that $\cW(A^{2n+2}_{2k+1})$ cannot be a coproduct by using Lemma \ref{lemma decomposable}.

\subsection{Decomposability of dg categories}
\label{subsection decomposable lemme}

Let $k$ be the coefficient ring. 
For a $k$-linear dg category $\cC$, we write $\hom^*_{\cC}(A,B)$ for the morphism complex where $A,B$ are objects in $\cC$, and we write $\Hom^*_{\cC}(A,B)$ for its cohomology. 
Let $\Tw(\cC)$ be the dg category of (one-sided) twisted complexes in $\cC$.
For the above definitions, see \cite{Bondal-Kapranov}.

\begin{definition}
	\label{def generator}
	Let $\cC$ be a dg category. 
	We say that {\em the set of objects $\{D_i\, |\, i\in I\}$ of $\Tw(\cC)$ is a generating set of $\cC$} if
	\[\Tw(\cC)\simeq\Tw(\cD),\]
	up to quasi-equivalence, where $\cD$ is the full dg subcategory of $\Tw(\cC)$ with the set objects $\{D_i\,|\,i\in I\}$, and $I$ is an indexing set.
\end{definition}

\begin{proposition}
	We have the quasi-equivalence
	\[\Tw(\cD\sqcup\cE)\simeq\Tw(\cD)\oplus\Tw(\cE)\]
	where the objects of $\Tw(\cD)\oplus\Tw(\cE)$ are the direct sums of the objects of $\Tw(\cD)$ and $\Tw(\cE)$.
\end{proposition}

\begin{proof}
	If $F\in\Tw(\cD\sqcup\cE)$, there exists a twisted complex for $F$ consisting of the objects of $\cD$ and $\cE$. Since there is no nonzero morphisms between $\cD$ and $\cE$, we can rearrange the twisted complex in such a way that all the objects of $\cD$ are on the left, and all the objects of $\cE$ are on the right. This means that
	\[F\simeq\Cone(D\xrightarrow{0}E)\simeq D[1]\oplus E\in\Tw(\cD)\oplus\Tw(\cE)\]
	for some $D\in\Tw(\cD)$ and $E\in\Tw(\cE)$. Similarly, the converse is true.
\end{proof}

\begin{definition}
	Let $\cC$ be a dg category.
	We call $\Tw(\cC)$ \textit{decomposable} if $\cC$ is pretriangulated equivalent to a coproduct of nonempty dg categories, i.e.
	\[\Tw(\cC)\simeq\Tw(\cD\sqcup\cE)\simeq\Tw(\cD)\oplus\Tw(\cE)\]
	up to quasi-equivalence for some nonempty dg categories $\cD$ and $\cE$.
	Otherwise, $\Tw(\cC)$ is called \textit{indecomposable}.
	We call $\Tw(\cC)$ decomposable with $n$-components, if
	\[\Tw(\cC)\simeq\Tw(\cC_1\sqcup\cC_2\sqcup\ldots\sqcup\cC_n)\simeq\Tw(\cC_1)\oplus\Tw(\cC_2)\oplus\ldots\oplus\Tw(\cC_n)\]
	up to quasi-equivalence for some nonempty dg categories $\cC_1,\cC_2,\ldots,\cC_n$ such that $\Tw(\cC_i)$ is indecomposable for all $i$. This is well-defined.
\end{definition}

\begin{lemma}
	\label{lemma decomposable}
	Let $\cC$ be a dg category. Assume that $\{F_i\, |\, i\in I\}$ is a generating set for $\cC$ satisfying
	\[\Hom^*(F_i,F_j)\neq 0 \text{  or  } \Hom^*(F_j,F_i)\neq 0,\]
	for all $i,j\in I$, and
	\[\Hom^0(F_\ell,F_\ell)=k\]
	for all $\ell\in I$. Then $\Tw(\cC)$ is indecomposable.
\end{lemma}

\begin{proof}
	We will prove the contrapositive. Assume $\Tw(\cC)$ is decomposable, and
	\[\Hom^*(F_i,F_j)\neq 0 \text{  or  } \Hom^*(F_j,F_i)\neq 0,\]
	for all $i,j\in I$. We will show that $\Hom^0(F_\ell,F_\ell)\neq k$ for some $\ell\in I$.
	
	Since $\Tw(\cC)$ is decomposable, there exist nonempty dg categories $\cD$ and $\cE$ such that
	\[\Tw(\cC)\simeq\Tw(\cD\sqcup\cE)\simeq\Tw(\cD)\oplus\Tw(\cE)\]
	up to quasi-equivalence. For all $i\in I$, since $F_i\in\Tw(\cC)$, we have
	\[F_i\simeq D_i\oplus E_i\]
	for some $D_i\in\Tw(\cD)$ and $E_i\in\Tw(\cE)$.
	We note that 
	\begin{gather*}
		\Hom^*(D_i,E_j)=0 \text{  and  } \Hom^*(E_j,D_i)=0,
	\end{gather*}
	for any $i,j\in I$.
	
	\textbf{Case 1:} Assume there exists $\ell\in I$ such that $D_\ell\not\simeq 0$ and $E_\ell\not\simeq 0$. Then,
	\begin{align*}
		\hom^*(F_\ell,F_\ell)&\simeq \hom^*(D_\ell\oplus E_\ell,D_\ell\oplus E_\ell)\\
		&\simeq \hom^*(D_\ell,D_\ell)\oplus \hom^*(E_\ell,E_\ell)\oplus \hom^*(D_\ell,E_\ell)\oplus \hom^*(E_\ell,D_\ell).
	\end{align*}
	Since $D_\ell\not\simeq 0$ and $E_\ell\not\simeq 0$, we have
	\begin{gather*}
		0\neq [1_{D_\ell}]\in\Hom^0(D_\ell,D_\ell), \text{ and  } 0\neq [1_{E_\ell}]\in\Hom^0(E_\ell,E_\ell) .
	\end{gather*}
	Hence, we get
	\[k\{[1_{D_\ell}]\}\oplus k\{[1_{E_\ell}]\}\subset\Hom^0(F_\ell,F_\ell),\]
	and consequently,
	\[\Hom^0(F_\ell,F_\ell)\neq k,\]
	which is what we wanted to show.
	
	\textbf{Case 2:} Assume Case 1 is not true. Then either $D_i\simeq 0$ for all $i\in I$ or $E_i\simeq 0$ for all $i\in I$. If not, there exist $i,j\in I$ with $i\neq j$ such that
	\[F_{i}\simeq D_{i}\not\simeq 0,\]
	and
	\[F_{j}\simeq E_{j}\not\simeq 0.\]
	This implies that
	\begin{gather*}
		\Hom^*(F_{i},F_{j})=0, \text{  and } \Hom^*(F_{j},F_{i})=0,
	\end{gather*}
	which contradicts with the assumption on $\{F_i\,|\, i\in I\}$ at the start of the proof. Thus, without loss of generality, we can assume $E_i\simeq 0$ for all $i\in I$. Then $F_i\simeq D_i\in\Tw(\cD)$ for all $i\in I$. Since $\cE$ is nonempty, there exists $E\in\cE$ such that $E\not\simeq 0$. However,
	\[\Hom^*(F_i,E)=0,\]
	for all $i\in I$. This means that $E\simeq 0$, since $\{F_i\,|\,i\in I\}$ generates $\cC$ and $E\in\Tw(\cC)$. This is a contradiction. Hence, Case 1 holds.
\end{proof}

\begin{remark}
	\label{rmk decomposable}
	Before going further, we remark that Lemma \ref{lemma decomposable} gives us a criterion for determining whether $\Tw(\cC)$ is decomposable or not. We will use this in Section \ref{subsection different wrapped Fukaya categories}, in order to show that $\cW(\X)$ is indecomposable.
\end{remark}

\subsection{Distinguishing wrapped Fukaya categories of $\X$ and $\Y$}
\label{subsection different wrapped Fukaya categories}
We prove the following proposition in Section \ref{subsection different wrapped Fukaya categories}.

\begin{proposition}
	\label{prop not symplectomorphic}
	The Weinstein manifolds $\X$ and $\Y$ are not exact deformation equivalent. 
\end{proposition}

We prove Proposition \ref{prop not symplectomorphic} by using Lemma \ref{lemma decomposable}. 
More precisely, we will prove that $\cW(\X)$ is not decomposable by Lemma \ref{lemma decomposable}.
In order to do this, we would like to point out that Lemma \ref{lemma decomposable} works when $\cC$ is an $A_{\infty}$-category. 
The reason is as follows: If $\cC$ is an $A_{\infty}$-category, we can consider the $A_{\infty}$-Yoneda embedding
\[\mathcal{Y}\colon\cC\hookrightarrow\Mod\,\cC\]
where $\Mod\,\cC$ is the dg category of $A_{\infty}$-modules over $\cC$.
See \cite{Seidel-bluebook}. 
Hence, $\cC$ is quasi-equivalent to the dg category $\mathcal{Y}(\cC)$. 
Consequently, Lemma \ref{lemma decomposable} applies on $\mathcal{Y}(\cC)$, and hence on $\cC$ since $\Hom^*(A,B)$ is the same in both categories for any objects $A$ and $B$.

\begin{proof}[Proof of Proposition \ref{prop not symplectomorphic}]
	We note that $\X \simeq A^{2n+2}_{2k+1}$ admits a Lefschetz fibration such that 
	\begin{itemize}
		\item the regular fiber is $T^*S^n$, and
		\item the number of singular values are $(k+1)$.
	\end{itemize} 
	This Lefschetz fibration is described in Section \ref{subsection wrapped Fukaya of X}.
	
	From the above Lefschetz fibration, we can choose a set of Lefschetz thimbles $\{T_1, \cdots, T_{k+1}\}$.
	Then, it is well-known that the set of Lefschetz thimbles is a generating set of $\cW(\X)$.
	
	Lemma \ref{lemma Ak wrapped Fukaya} which appears below says that 
	\[\Hom^*(T_i,T_j) \neq 0,\]
	for any $i \neq j$. 
	Thus, if $\cW(\X)$ is decomposable, there is $\ell \in \{1, \cdots, k+1\}$ such that 
	\[\Hom^0(T_\ell, T_\ell) \neq k,\]
	by Lemma \ref{lemma decomposable} and Remark \ref{rmk decomposable}.
	
	In the current literature, $\Hom(T_\ell,T_\ell)$ is already computed for any $\ell$. 
	We refer the reader to \cite{Bae-Kwon} or \cite{Lekili-Ueda}. 
	Since our case $\X \simeq A^{2n+2}_{2k+1}$ satisfies the condition of \cite[Proposition 1.5]{Bae-Kwon}, for all $\ell$, 
	\[\Hom^0(T_\ell, T_\ell) = k,\]
	where the base ring $k$ is $\mathbb{Z}_2$. 
	
	This concludes that $\cW(\X)$ cannot be a coproduct of two non-empty categories.
	However, Lemma \ref{lemma wrapped Fukaya of Y} says that $\cW(\Y)$ can be written as a coproduct.
	It concludes that $\cW(\X)$ and $\cW(\Y)$ are not equivalent.
	Thus, $\X$ and $\Y$ are not equivalent as Weinstein manifolds.
\end{proof}

\begin{lemma}
	\label{lemma Ak wrapped Fukaya}
	Let $T_1, \cdots, T_{k+1}$ be Lefschetz thimbles defined above. 
	Then,
	\[\Hom^*(T_i, T_j)\neq 0.\]
\end{lemma}
\begin{proof}
	We use the index-positivity argument described in \cite{Bae-Kwon}, in order to show that $Hom^*(T_i, T_j) \neq 0$. 
	Let us use a following explicit description of $A_{2k+1}^{2n+2} \simeq X^{2n+2}_k$, and of the Lefschetz fibration on $A_{2k+1}^{2n+2}$
	\begin{gather*}
		\pi: A^{2n+2}_{k+1}=\{x_0^{k+1}+x_1^2+\cdots+x_{n+1}^2=1\} \to \mathbb{C}, \\
		(x_0, \cdots, x_{n+1}) \mapsto x_0.
	\end{gather*}
	
	The critical values of the fibration are $\xi^i_{k+1} \hskip 0.1cm (i=0, \ldots, k)$ where $\xi^i_{k+1} = e^{\frac{2 \pi i}{k+1}} \in \mathbb{C}$.
	Let 
	\[\Gamma_i = \{r \xi_{k+1}^{i}: r\geq1\}, \hskip 0.1cm \text{for  } i=0, \cdots, k,\] 
	denote rays in $\mathbb{C}$, emanating from each critical points. 
	The Lagrangians $T_i$ we are considering are thimbles over each $\Gamma_i$;
	\[T_i := \{(r \xi_{k+1}^{i}, x_1, \ldots, x_{n+1}) :  x_1^2 \cdots +x_{n+1}^2 =1-r^{k+1}, \hskip 0.1cm x_i \in \sqrt {-1} \mathbb R, \hskip 0.1cm i\geq1, \hskip 0.1cm r\geq1\}.\]
	We note that $Hom^*(T_i,T_j)$ is defined geometrically, by using the relation of Hamiltonian chords from $T_i$ to $T_j$. 
	
	 A sequence of admissible Hamiltonians $H_\ell$ we are using is of the form $h_\ell \circ \pi$, where $h_\ell : \mathbb C \to \mathbb R$ is a smooth function such that 
	 \begin{itemize}
	 	\item $h_\ell$ only depends on $\vert z\vert$,
	 	\item $h_\ell$ is $C^2$-small on a compact region, and
	 	\item $h_\ell(z)=\ell \vert z\vert^2$ when $\vert z\vert \gg 1$.
	 \end{itemize}  
	
	Hamiltonian chords from $T_i$ to $T_j$ of $H_\ell$ comes in families because the set of vanishing cycles of the thimbles have rotational symmetries. 
	Whenever we have a Hamiltonian chords of $h_\ell$ from $\Gamma_i$ to $\Gamma_j$ on the base, we get an $S^n$-family of corresponding chords of $H_\ell$ from $T_i$ to $T_j$. 
	Using a Morse-Bott type perturbation, we get two non-degenerate Hamiltonian chords coming from Morse homology of a sphere $S^n$. 
	We label the generators of the Floer complex as 
	\[\{\gamma_{m, min}, \gamma_{m, max}: m\in \mathbb Z_{\geq0}\},\] 
	where a non-negative number $m$ indicates the winding number of the corresponding base chord, and where min/max denotes the generators corresponding to a fundamental/point class of $H_{*, \mathrm{Morse}}(S^n)$. 
	In this setup, we have
	\begin{align*}
	\mu_{\mathrm{Maslov}}(\gamma_{m+1, max})-\mu_{\mathrm{Maslov}}(\gamma_{m, min})&=(n-1)(k+1)+2 \geq 4,\\
	\mu_{\mathrm{Maslov}}(\gamma_{m, min})-\mu_{\mathrm{Maslov}}(\gamma_{m, max})&= n \geq 2,
	\end{align*}
	when $n\geq2$ and $k\geq1$. 
	See Proposition 1.2 in \cite{Bae-Kwon} for more details.
	In particular, we conclude that there is no non-trivial Floer differential hitting $\gamma_{0, max}$. (In fact, the spectral sequence degenerates at this page whenever $k\geq 2$.) 
	This implies that $Hom^*(T_i, T_j)\neq 0$. 
\end{proof}

\section{Diffeomorphism between $\XX$ and $\YY$}
\label{section diffeo}
In Section \ref{section diffeo}, we prove Proposition \ref{prop diffeo} which studies the diffeomorphic classes of $\X$ and $\Y$. 
\begin{proposition}
	\label{prop diffeo}
	If $n = 2$, the Weinstein manifolds $\X$ and $\Y$ are diffeomorphic.
	If $n \geq 4$ is even, then $\XX$ and $\YY$ are diffeomorphic.
\end{proposition}

We note that the main idea of Proposition \ref{prop diffeo} is given in \cite[Proposition 4.0.1]{Maydanskiy} and \cite{Maydanskiy-Seidel15}.
For more details, we refer the reader to them.

\begin{proof}[Proof of Proposition \ref{prop diffeo}]
In order to prove this, let us recall the construction of $\X$ and $\Y$. 

In Section \ref{subsection construction}, $\X$ and $\Y$ are given as total spaces of abstract Lefschetz fibrations in \eqref{eqn X} and \eqref{eqn Y}. 
Section \ref{subsection wrapped Fukaya of Y} gives more detailed construction of $\Y$, which is a Weinstein domain obtained by attaching critical handles $H_1, \cdots, H_{2k+3}$ to the subcritical part $A^{2n}_2 \times \mathbb{D}_1^2$. 
Similarly, $\X$ is equivalent to the Weinstein domain obtained by attaching critical handles $I_1, \cdots, I_{2k+3}$ to the same subcritical part $A^{2n}_2 \times \mathbb{D}_1^2$. 
For more detail on the constructions, we refer the reader to \cite[Lemma 16.9]{Seidel-bluebook}, \cite[Section 8]{Bourgeois-Ekholm-Eliashberg}. 
See also \cite{Gompf-Stipsicz}.

From \eqref{eqn X} and \eqref{eqn Y}, without loss of generality, one could assume that $I_i$ and $H_i$ are attached along the same Legendrian of $\partial \left(A^{2n}_2 \times \mathbb{D}_1^2\right)$, if $i \neq 2k+2$. 
Thus, in order to prove Proposition \ref{prop diffeo}, it is enough to consider the attaching Legendrian spheres of $I_{2k+2}$ and $H_{2k+2}$. 

The attaching Legendrian spheres of $I_{2k+2}$ and $H_{2k+2}$ are induced by Lagrangian spheres $(\tau_\alpha)^{2k}(\beta)$ and $\beta$. 
More precisely, these Lagrangian spheres induce Legendrian spheres with the canonical formal Legendrian structures as defined in \cite{Murphy}.
We note that when one attaches a critical handle along a Legendrian sphere, the formal Legendrian structure on the Legendrian sphere determines a diffeomorphism class of the resulting space.
This is because the formal Legendrian structure affects on the framing of the attached critical handle. 

In order to prove Proposition \ref{prop diffeo}, we need to compare the Legendrian spheres induced from $\beta$ and $(\tau_\alpha)^{2k}(\beta)$ together with their formal Legendrian structures for $n =2$, and to compare the Legendrian spheres induced from $\beta$ and $(\tau_\alpha)^{4k}(\beta)$ together with their formal Legendrian structures for even $n \geq 4$.
\vskip0.2in

\noindent{\em Construction of smooth isotopy}: 
First, we construct a smooth Lagrangian isotopy between $\beta$ and $(\tau_\alpha)^2(\beta)$ inside $A_2^{2n}$ for any even $n$. 
Then, it induces an isotopy between $\beta$ and $(\tau_\alpha)^{2k}(\beta)$, naturally. 
This construction is originally given in \cite[Section 5]{Maydanskiy}. 

By using the Lefschetz fibration $\rho$ on $A^{2n}_2$, $\beta$ (resp.\ $\tau_\alpha^2(\beta)$) is given as a union of vanishing cycles in fibers of $\rho$ along a curve in $\mathbb{C}$. 
Let $\gamma_1$ and $\gamma_2$ be the curves corresponding to $\beta$ and $\tau_\alpha^2(\beta)$, respectively. 
See Figure \ref{figure curves}.

\begin{figure}[h]
	\centering
\begingroup%
  \makeatletter%
  \providecommand\color[2][]{%
    \errmessage{(Inkscape) Color is used for the text in Inkscape, but the package 'color.sty' is not loaded}%
    \renewcommand\color[2][]{}%
  }%
  \providecommand\transparent[1]{%
    \errmessage{(Inkscape) Transparency is used (non-zero) for the text in Inkscape, but the package 'transparent.sty' is not loaded}%
    \renewcommand\transparent[1]{}%
  }%
  \providecommand\rotatebox[2]{#2}%
  \newcommand*\fsize{\dimexpr\f@size pt\relax}%
  \newcommand*\lineheight[1]{\fontsize{\fsize}{#1\fsize}\selectfont}%
  \ifx\svgwidth\undefined%
    \setlength{\unitlength}{170.07874016bp}%
    \ifx\svgscale\undefined%
      \relax%
    \else%
      \setlength{\unitlength}{\unitlength * \real{\svgscale}}%
    \fi%
  \else%
    \setlength{\unitlength}{\svgwidth}%
  \fi%
  \global\let\svgwidth\undefined%
  \global\let\svgscale\undefined%
  \makeatother%
  \begin{picture}(1,1)%
    \lineheight{1}%
    \setlength\tabcolsep{0pt}%
    \put(0,0){\includegraphics[width=\unitlength,page=1]{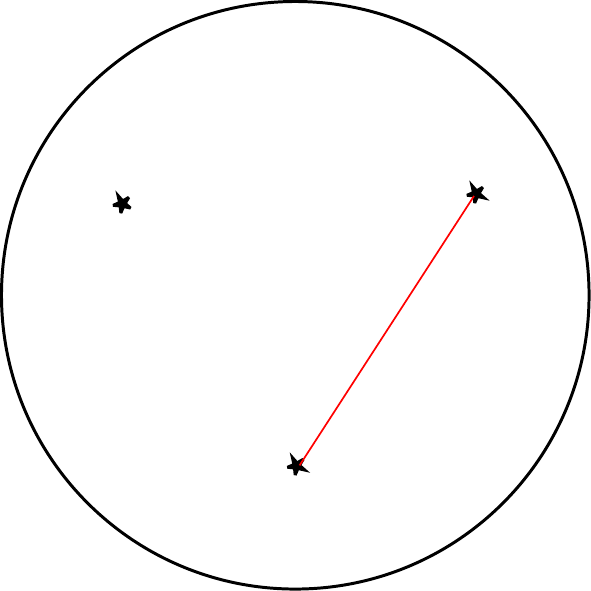}}%
    \put(0.68282248,0.40419123){\makebox(0,0)[lt]{\lineheight{1.25}\smash{\begin{tabular}[t]{l}$\gamma_1 = \rho(\beta)$\end{tabular}}}}%
    \put(0,0){\includegraphics[width=\unitlength,page=2]{curves.pdf}}%
    \put(0.27664639,0.79752686){\makebox(0,0)[lt]{\lineheight{1.25}\smash{\begin{tabular}[t]{l}$\gamma_2$\end{tabular}}}}%
  \end{picture}%
\endgroup%
		
	\caption{The red and blue curves on base of $\rho$ are $\gamma_1$ and $\gamma_2$, respectively.}
	\label{figure curves}
\end{figure}

We note that the regular fiber of $\rho$ is $T^*S^{n-1}$, and every vanishing cycle of $\rho$ is the zero section. 
From the structure of $T^*S^{n-1}$, when $n$ is even, we can lift the zero section in a Lagrangian isotopic way, but not Hamiltonian isotopic way.
This is because $S^{n-1}$ admits a non-vanishing vector field. 
The lift for the case of $n=2$ is given in Figure \ref{figure lift}, $a)$.
We note that we can control how much the vanishing cycle is lifted. 

\begin{figure}[h]
	\centering
\begingroup%
  \makeatletter%
  \providecommand\color[2][]{%
    \errmessage{(Inkscape) Color is used for the text in Inkscape, but the package 'color.sty' is not loaded}%
    \renewcommand\color[2][]{}%
  }%
  \providecommand\transparent[1]{%
    \errmessage{(Inkscape) Transparency is used (non-zero) for the text in Inkscape, but the package 'transparent.sty' is not loaded}%
    \renewcommand\transparent[1]{}%
  }%
  \providecommand\rotatebox[2]{#2}%
  \newcommand*\fsize{\dimexpr\f@size pt\relax}%
  \newcommand*\lineheight[1]{\fontsize{\fsize}{#1\fsize}\selectfont}%
  \ifx\svgwidth\undefined%
    \setlength{\unitlength}{340.15748031bp}%
    \ifx\svgscale\undefined%
      \relax%
    \else%
      \setlength{\unitlength}{\unitlength * \real{\svgscale}}%
    \fi%
  \else%
    \setlength{\unitlength}{\svgwidth}%
  \fi%
  \global\let\svgwidth\undefined%
  \global\let\svgscale\undefined%
  \makeatother%
  \begin{picture}(1,0.5)%
    \lineheight{1}%
    \setlength\tabcolsep{0pt}%
    \put(0,0){\includegraphics[width=\unitlength,page=1]{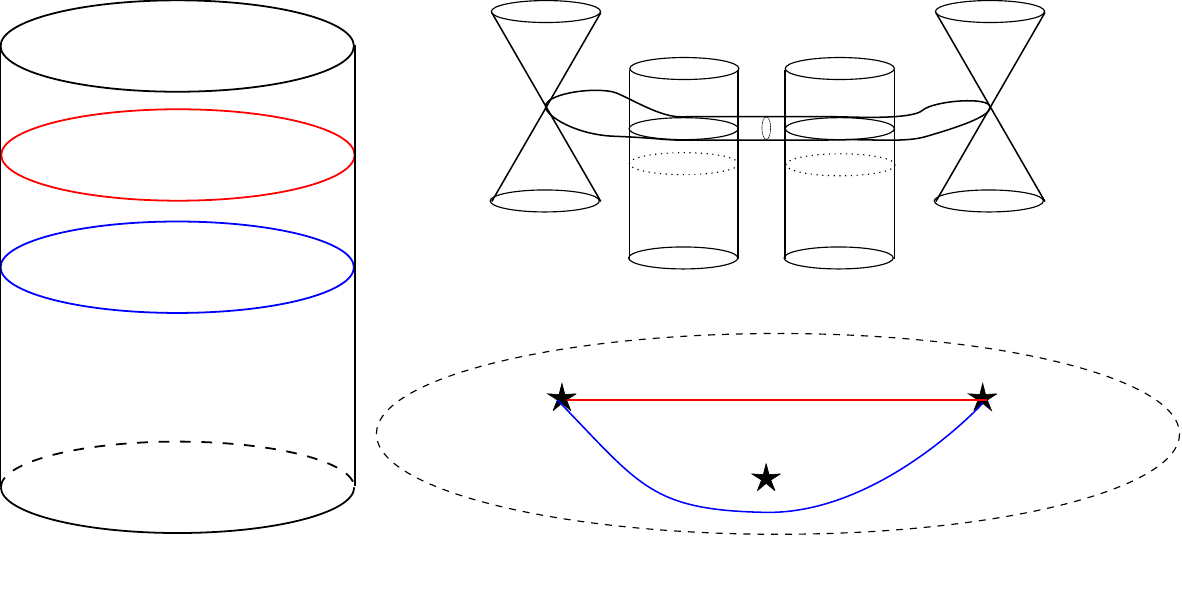}}%
    \put(0.14001753,0.01945877){\makebox(0,0)[lt]{\lineheight{1.25}\smash{\begin{tabular}[t]{l}$a).$\end{tabular}}}}%
    \put(0.63725478,0.01520618){\makebox(0,0)[lt]{\lineheight{1.25}\smash{\begin{tabular}[t]{l}$b).$\end{tabular}}}}%
  \end{picture}%
\endgroup%
		
	\caption{$a).$ The cylinder is the regular fiber $T^*S^1$, the blue curve is the vanishing cycle, i.e., the zero section, and the red curve is a lifted vanishing cycle.
		$b).$ The Lagrangian sphere isotopic to $\tau_\alpha^2(\beta)$ is the union of lifted vanishing cycles. In the regular fibers, the dotted circles are vanishing cycles without lifts.
		The curves $\gamma_1 = \rho(\beta)$ and $\gamma_2$ are colored red and blue on the base. }
	\label{figure lift}
\end{figure}	

After that, we consider the union of lifted vanishing cycles inside $\rho^{-1}(p)$.
The union runs over $\{p \in \gamma_2\}$. 
The union is over the fibers projected on $\gamma_2$.
How much the vanishing cycle in $\pi^{-1}(p)$ lifted is determined by the distance between $p$ and $\rho(\beta)$.
If a point $p$ on the base is close enough to $\rho(\beta)$, then we do not lift the vanishing cycle in $\rho^{-1}(p)$.
And, if a point $p$ is not close to $\rho(\beta)$, then we lift the vanishing cycles in $\rho^{-1}(p)$.
If $p_1$ is further from $\rho(\beta)$ than $p_2$, then we lift more the vanishing cycle in $\rho^{-1}(p_1)$ than that in $\rho^{-1}(p_2)$.  
Then, this union of lifted vanishing cycles is a Lagrangian sphere which is isotopic to $\tau_\alpha^2(\beta)$. 
Figure \ref{figure lift}, $b)$ describes the case of $n=2$. 

On the base of $\rho$, two curves $\gamma_1$ and $\gamma_2$ are isotopic to each other. 
Along the isotopy of curves connecting $\gamma_1$ and $\gamma_2$, we can construct a smooth family of Lagrangian spheres, whose each member is a union of lifted vanishing cycles.
Also, we consider the lifted vanishing cycles, the family of Lagrangian spheres does not touch the singular point.
Then, this family gives us an isotopy connecting $\beta$ and $\tau_\alpha^2(\beta)$. 
\vskip0.2in

\noindent{\em Formal Legendrian structures}:
Notice that for a given Legendrian sphere of dimension $n$, the set of formal Legendrian structures is given by $\pi_{n+1}(V_{n,2n+1},U_n)$, where $V_{n,2n+1}$ stands for the {\em Stiefel manifold}. 
For more detail, see \cite[Appendix]{Murphy}. 

For the case of $n =2$, \cite[Proposition A.4]{Murphy} says that all formal embeddings of $S^2$ are formally Legendrian isotopic, or equivalently, $\pi_3(V_{2,5},U_2)$ is trivial. 
Thus, the isotopy between $\beta$ and $(\tau_\alpha)^{2k}(\beta)$ which is constructed above induces a smooth isotopy connecting the induced Legendrian spheres with formal Legendrian structures. 
This prove the first half of Proposition \ref{prop diffeo}.

For the case of even $n \geq 4$, \cite{Maydanskiy-Seidel15} shows that the constructed isotopy between $\beta$ and $(\tau_\alpha)^4(\beta)$ induces a smooth isotopy connecting the induced Legendrian spheres with formal Legendrian structures.
In order to show this, the authors recalled that 
\begin{gather}
	\label{eqn homotopy group of stiefel}
	\pi_{n+1}(V_{n,2n+1},U_n) \simeq \mathbb{Z}/2.
\end{gather}

If we measure the formal Legendrian structures given on $\beta$ and $(\tau_\alpha)^2(\beta)$, the difference must lies in the group $\mathbb{Z}/2$ by Equation \eqref{eqn homotopy group of stiefel}. 
Therefore we might get a non-trivial element. 
However, when we consider $\beta$ and $\tau_\alpha^4(\beta)$, the difference is the trivial element in because we take the same operation twice to cancel out the difference inside $\pi_{n+1}(V_{n,2n+1},U_n)$.
This completes the proof. 
\end{proof}

We are now ready to prove Theorem \ref{thm main}.
\begin{theorem}
	\label{thm main}
	Let $\X$ and $\Y$ be symplectic manifolds constructed in Section \ref{section construction}.
	Then, for $n=2$ (resp.\ even $n \geq 4$)  the pair $(\X,\Y)$ (resp.\ $(\XX,\YY)$) satisfies the followings:
	\begin{itemize}
		\item $\X$ and $\Y$ (resp.\ $\XX$ and $\YY$) are diffeomorphic,
		\item $\X$ and $\Y$ are not exact deformation equivalent, and 
		\item $\cW(\X)$ and $\cW(\Y)$ are not vanishing. 
	\end{itemize}
\end{theorem}
\begin{proof}
	Proposition \ref{prop not symplectomorphic} and Proposition \ref{prop diffeo} prove Theorem \ref{thm main}. 
\end{proof}

\begin{remark}
	\label{rmk notation}
	Before going further, we remark that in Section \ref{section exotic familes of plumbings with names}, we extend the exotic pair $(\XX, \YY)$ to diffeomorphic families of different Weinstein manifolds. 
	However, in the next section, we use different notation $P_{n+1}(T_{4k}^1)$ (resp.\ $Q_{2k+1}^{2n+2}$) in order to denote $\XX$ (resp.\ $\YY$).
	Since Proposition \ref{prop diffeo} is a part of Lemmas \ref{lemma diffeomorphic families 1} and \ref{lemma diffeomorphic families 2}, it would seem more reasonable to use $P_{n+1}(T_{4k}^1)$ and $Q_{2k+1}^{2n+2}$, instead of $\XX$ and $\YY$.
	Moreover, instead of $\XX$, we can directly say that it is the Milnor fiber of $A_{4k+1}$-singularity, as proven in Lemma \ref{lemma strucutre of X}. 
	In this remark, we would like to clarify the reason of using the notations $\XX$ and $\YY$.
	
	We would like to point out that $\X$ and $\Y$ are different as Weinstein manifolds since $\Y$ has a decomposable Fukaya category, but $\X$ does not. 
	Similarly, if another Weinstein manifold $Z$ has non-decomposable Fukaya category, then $Z$ and $\Y$ are different. 
	Based on this, for example, we can expect that the arguments in Sections \ref{section construction} -- \ref{section diffeo} distinguish $\Y$ and the Milnor fiber of $D_{2k+1}$-type as Weinstein manifolds even if they are diffeomorphic to each other. 
	Actually, this corresponds to the fourth families in Theorems \ref{thm exotic families} and \ref{thm exotic families 2}.
	Moreover, we can consider to replace $\X$ with another plumbing space in the same way. 

	By using a letter $X$ which is the conventional letter for a variable, we would like to implicitly mention that one can choose different space for the position of $\X$ in the exotic pair, i.e., $\Y$.
\end{remark}
 
\section{The first generalization}
\label{section a simple extension}
In Sections \ref{section construction} -- \ref{section diffeo}, we constructed exotic pairs of Weinstein manifolds. 
Section \ref{section a simple extension} extends the construction to families of diffeomorphic, but different Weinstein manifolds. 

In order to prepare the extension, we introduce notation first. 
Let $A_m^{2n}$ be the $A_m$ type plumbing of $T^*S^n$ as defined in Definition \ref{def A plumbig}.
Then, there are Lagrangian spheres which are zero sections of $T^*S^n$. 
Let {\em $\alpha_1, \cdots, \alpha_m$} denote the Lagrangian spheres so that 
$\alpha_i$ and $\alpha_{i+1}$ intersect at one point. 
Since $A_m^{2n}$ is obtained by plumbing a cotangent bundle of $S^n$ to $A_{m-1}^{2n}$, the following relations make sense by abuses of notation. 
\[\alpha_1, \cdots, \alpha_{m-1} \subset A_{m-1}^{2n} \subset A_m^{2n}.\]

For a Lagrangian sphere $S$, there is a Dehn twist along $S$. 
Let $\tau_S$ denote the Dehn twist along $S$. 

With the above notation, one can define Definition \ref{def Z}.
\begin{definition}
	\label{def Z}
	For natural numbers $i_1, \cdots, i_k \in \mathbb{N}$, let {\em $\Z_{(i_1;i_2;\cdots;i_k)}$} denote the total space of the following abstract Lefschetz fibration:
	\begin{gather}
		\label{eqn Z}
		(A_k^{2n}; \alpha_1, \cdots, \alpha_1, \alpha_2, \cdots, \alpha_2, \alpha_3, \cdots, \alpha_{k-1}, \alpha_k, \cdots, \alpha_k),
	\end{gather}
	where the number of $\alpha_j$ in \eqref{eqn Z} is $i_j+1$.
\end{definition}

It is easy to check that $A_m^{2n+2}$ is equivalent to $\Z_{(m)}$ as a Weinstein manifold by the existence of a Lefschetz fibration $\rho$ given below of Definition \ref{def A plumbig}.
Similarly, $Y_m^{2n+2}$ is equivalent to $\Z_{(2m;1)}$. 

\begin{lemma}
	\label{lemma diffeomorphic family of Z}
	\mbox{}
	\begin{enumerate}
		\item If $n=2$, then two Weinstein manifolds 
		\[\Z_{(i_1;i_2;\cdots;i_{k-2};2i_{k-1};i_k)} \text{  and  } \Z_{(i_1;i_2;\cdots;i_{k-2};2i_{k-1}+i_k)}\]
		are diffeomorphic to each other.
		\item If $n \geq 4$ is even, then two Weinstein manifolds \[\Z_{(i_1;i_2;\cdots;i_{k-2};4i_{k-1};i_k)} \text{  and  } \Z_{(i_1;i_2;\cdots;i_{k-2};4i_{k-1}+i_k)}\]
		are diffeomorphic to each other.
	\end{enumerate}
\end{lemma}
\begin{proof}
	We prove the second case of Lemma \ref{lemma diffeomorphic family of Z}.
	The first case can be proven in the same way. 
	
	By definition, $\Z_{(i_1;i_2;\cdots;4i_{k-1}+i_k)}$ is the total space of 
	\[(A_{k-1}^{2n}; \alpha_1, \cdots, \alpha_1, \alpha_2, \cdots, \alpha_2, \alpha_3, \cdots, \alpha_{k-2}, \alpha_{k-1}, \cdots, \alpha_{k-1}).\]
	Then, by the stabilization given in Section \ref{subsection equivalent abstract Lefschetz fibrrations}, one can modify the above and can obtain 
	\[(A_k^{2n}; \alpha_1, \cdots, \alpha_1, \alpha_2, \cdots, \alpha_2, \alpha_3, \cdots, \alpha_{k-2}, \alpha_{k-1}, \cdots, \alpha_{k-1}, \alpha_k, \alpha_{k-1}).\]
	By a Hurwitz move, one can move $\alpha_k$ to the right, i.e., one has 
	\begin{gather}
		\label{eqn middle1}
	(A_k^{2n}; \alpha_1, \cdots, \alpha_1, \alpha_2, \cdots, \alpha_2, \alpha_3, \cdots, \alpha_{k-2}, \alpha_{k-1}, \cdots, \alpha_{k-1}, \tau_{\alpha_{k-1}}(\alpha_k)).
	\end{gather}
	We note that in \eqref{eqn middle1}, the number of $\alpha_{k-1}$ is $(4i_{k-1} + i_k +1)$. 
	
	One can move the last $\alpha_{k-1}$ to the right by a Hurewitz move. 
	Then, $\alpha_{k-1}$ becomes 
	\[\tau_{\tau_{\alpha_{k-1}(\alpha_k)}}(\alpha_{k-1}) = \alpha_k.\]
	One can easily check the equality by using the Lefschetz fibration $\rho$ of $A_k^{2n}$, which is given right below of Definition \ref{def A plumbig}.
	Then, as a result of the Hurewitz move, one obtains 
	\[(A_k^{2n}; \alpha_1, \cdots, \alpha_1, \alpha_2, \cdots, \alpha_2, \alpha_3, \cdots, \alpha_{k-2}, \alpha_{k-1}, \cdots, \alpha_{k-1}, \tau_{\alpha_{k-1}}(\alpha_k), \alpha_k).\]

	Similarly, one can operate the similar Hurwitz move $(i_k -1)$ times.
	Then, it gives 
	\begin{gather}
		\label{eqn middle2}
		(A_k^{2n}; \alpha_1, \cdots, \alpha_1, \alpha_2, \cdots, \alpha_2, \alpha_3, \cdots, \alpha_{k-2}, \alpha_{k-1}, \cdots, \alpha_{k-1}, \tau_{\alpha_{k-1}}(\alpha_k), \alpha_k, \cdots, \alpha_k).
	\end{gather} 
	We note that in \eqref{eqn middle2}, the numbers of $\alpha_{k-1}$ and $\alpha_k$ are $4 i_{k-1} + 1$ and $i_k$, respectively. 
	
	After that, one can move $\tau_{\alpha_{k-1}}(\alpha_k)$ to the left by operating Hurwitz moves $(4i_{k-1} +1)$ times. 
	Then, one obtains 
	\begin{gather}
		\label{eqn middle3}
		(A_k^{2n}; \alpha_1, \cdots, \alpha_1, \alpha_2, \cdots, \alpha_2, \alpha_3, \cdots, \alpha_{k-2}, \tau_{\alpha_{k-1}}^{-4i_{k-1}}(\alpha_k), \alpha_{k-1}, \cdots, \alpha_{k-1}, \alpha_k, \cdots, \alpha_k).
	\end{gather}
	
	By the argument in the proof of Proposition \ref{prop diffeo}, we know that the replacement of $\tau_{\alpha_{k-1}}^{-4i_{k-1}}(\alpha_k)$ by $\alpha_k$ in \eqref{eqn middle3} gives a diffeomorphic Weinstein manifold.
	In other words, the total space of
	\begin{gather}
		\label{eqn middle4}
		(A_k^{2n}; \alpha_1, \cdots, \alpha_1, \alpha_2, \cdots, \alpha_2, \alpha_3, \cdots, \alpha_{k-2}, \alpha_k, \alpha_{k-1}, \cdots, \alpha_{k-1}, \alpha_k, \cdots, \alpha_k).
	\end{gather} 
	is diffeomorphic to the original Weinstein manifold $\Z_{(i_1;\cdots;i_{k-2};4i_{k-1}+i_k)}$.
	
	Finally, one can move the $\alpha_k$ which is located in the middle of \eqref{eqn middle4} to the first place by Hurwitz moves. 
	Since $\tau_i(\alpha_k) = \alpha_k$ when $i \leq k-2$, one has  
	\[(A_k^{2n}; \alpha_k, \alpha_1, \cdots, \alpha_1, \alpha_2, \cdots, \alpha_2, \alpha_3, \cdots, \alpha_{k-2}, \alpha_{k-1}, \cdots, \alpha_{k-1}, \alpha_k, \cdots, \alpha_k).\]
	Also, by moving the first $\alpha_k$ to the last by a cyclic permutation given in Section \ref{subsection equivalent abstract Lefschetz fibrrations}, one obtains $\Z_{(i_1;i_2;\cdots;i_{k-2};4i_{k-1};i_k)}$.
	This completes the proof.
\end{proof}

	Lemma \ref{lemma diffeomorphic family of Z} induces Theorem \ref{thm exotic families of Z}.
	
	\begin{theorem}
		\label{thm exotic families of Z}
		\mbox{}
		\begin{enumerate}
			\item If $n=2$, the following family 
			\begin{gather}
				\label{eqn exotic family of Z1}
				\{\Z_{(2i_1+ \cdots 2i_{k-1} + i_k)}, \Z_{(2i_1;2i_2+\cdots 2i_{k-1} + 	i_k)}, \cdots, \Z_{(2i_1;2i_2;\cdots;2i_{k-2};2i_{k-1} + i_k)}, \Z_{(2i_1;2i_2;\cdots;2i_{k-2};2i_{k-1};i_k)}\}
			\end{gather}
			are diffeomorphic families of $k$ pairwise different Weinstein manifolds.
			\item If $n \geq 4$ is even, the following family 
			\begin{gather}
				\label{eqn exotic family of Z2}
				\{\Z_{(4i_1+ \cdots 4i_{k-1} + i_k)}, \Z_{(4i_1;4i_2+\cdots 4i_{k-1} + 	i_k)}, \cdots, \Z_{(4i_1;4i_2;\cdots;4i_{k-2};4i_{k-1} + i_k)}, \Z_{(4i_1;4i_2;\cdots;4i_{k-2};4i_{k-1};i_k)}\}
			\end{gather}
			are diffeomorphic families of $k$ pairwise different Weinstein manifolds.
		\end{enumerate}
	\end{theorem}
	\begin{proof}
		By Lemma \ref{lemma diffeomorphic family of Z} induces that the families in \eqref{eqn exotic family of Z1} and \eqref{eqn exotic family of Z2} are diffeomorphic families. 
		Thus, it is enough to show that the members of the families are pairwise different as Weinstein manifolds. 
		
		We note that by definition, $\Z_{(i_1;i_2;\cdots;i_k)}$ is a Weinstein manifold obtained by taking the end connected sum of $A_{i_1}^{2n+2}, \cdots, A_{i_k}^{2n+2}$, see the third item in Remark \ref{rmk facts}. 
		This induces that the wrapped Fukaya category of $\Z_{(i_1;i_2;\cdots;i_k)}$ can be written as a coproduct of $k$ nontrivial categories. 
		
		In the proof of Proposition \ref{prop not symplectomorphic}, we already checked that the wrapped Fukaya category of $A_k^{2n+2}$ could not be written as a coproduct of two or more nontrivial categories.
		Thus, the above argument implies that the members in \eqref{eqn exotic family of Z1} or \eqref{eqn exotic family of Z2} are pairwise different as Weinstien manifolds.
	\end{proof}
 
\section{The second generalization}
\label{section exotic familes of plumbings with names}
In Sections \ref{section construction} -- \ref{section diffeo}, we consider an exotic pair of Weinstein manifolds $(\X, \Y)$ (resp.\ $(\XX,\YY)$) for $k \geq 1$ and for $n = 2$ (resp.\ even $n \geq 4$). 
We distinguish them as Weinstein manifolds by using their wrapped Fukaya categories. 

In the current section, we extend the pairs to diffeomorphic families of different Weinstein manifolds.
The diffeomorphic families are listed at the end of Section \ref{subsection diffeomorphic families}.
We prove that the given families are diffeomorphic by using Lefschetz fibrations as similar to Section \ref{section diffeo}. 
However, we prove that they are different as Weinstein manifolds by using their symplectic cohomologies, which is a different method from the previous case.

\subsection{Construction of diffeomorphic families}
\label{subsection diffeomorphic families}
In Section \ref{subsection diffeomorphic families}, we give diffeomorphic families of Weinstein manifolds.
In order to do this, we need Definition \ref{def T_m^j}.

\begin{definition}
	\mbox{}
	\label{def T_m^j}
	\begin{enumerate}
		\item 	For any $m \in \mathbb{N}$, and any $ 1 \leq j \leq m$, let {\em $T_m^j$} denote the tree which is given in Figure \ref{figure T_m^j}. 
		\begin{figure}[h]
			\centering
\begingroup%
  \makeatletter%
  \providecommand\color[2][]{%
    \errmessage{(Inkscape) Color is used for the text in Inkscape, but the package 'color.sty' is not loaded}%
    \renewcommand\color[2][]{}%
  }%
  \providecommand\transparent[1]{%
    \errmessage{(Inkscape) Transparency is used (non-zero) for the text in Inkscape, but the package 'transparent.sty' is not loaded}%
    \renewcommand\transparent[1]{}%
  }%
  \providecommand\rotatebox[2]{#2}%
  \newcommand*\fsize{\dimexpr\f@size pt\relax}%
  \newcommand*\lineheight[1]{\fontsize{\fsize}{#1\fsize}\selectfont}%
  \ifx\svgwidth\undefined%
    \setlength{\unitlength}{226.77165354bp}%
    \ifx\svgscale\undefined%
      \relax%
    \else%
      \setlength{\unitlength}{\unitlength * \real{\svgscale}}%
    \fi%
  \else%
    \setlength{\unitlength}{\svgwidth}%
  \fi%
  \global\let\svgwidth\undefined%
  \global\let\svgscale\undefined%
  \makeatother%
  \begin{picture}(1,0.225)%
    \lineheight{1}%
    \setlength\tabcolsep{0pt}%
    \put(0,0){\includegraphics[width=\unitlength,page=1]{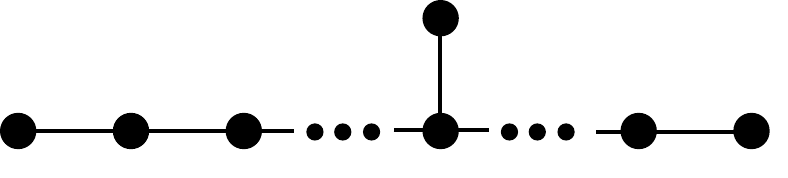}}%
    \put(0.01285024,0.00623575){\makebox(0,0)[lt]{\lineheight{1.25}\smash{\begin{tabular}[t]{l}$v_1$\end{tabular}}}}%
    \put(0.1554882,0.00623723){\makebox(0,0)[lt]{\lineheight{1.25}\smash{\begin{tabular}[t]{l}$v_2$\end{tabular}}}}%
    \put(0.29928831,0.00623723){\makebox(0,0)[lt]{\lineheight{1.25}\smash{\begin{tabular}[t]{l}$v_3$\end{tabular}}}}%
    \put(0.54762176,0.00623723){\makebox(0,0)[lt]{\lineheight{1.25}\smash{\begin{tabular}[t]{l}$v_j$\end{tabular}}}}%
    \put(0.78614463,0.00623723){\makebox(0,0)[lt]{\lineheight{1.25}\smash{\begin{tabular}[t]{l}$v_{m-1}$\end{tabular}}}}%
    \put(0.94503508,0.00623723){\makebox(0,0)[lt]{\lineheight{1.25}\smash{\begin{tabular}[t]{l}$v_m$\end{tabular}}}}%
    \put(0.58697062,0.18741217){\makebox(0,0)[lt]{\lineheight{1.25}\smash{\begin{tabular}[t]{l}$v_{m+1}$\end{tabular}}}}%
  \end{picture}%
\endgroup%
		
			\caption{Tree $T_m^j$.} 
			\label{figure T_m^j}
		\end{figure}
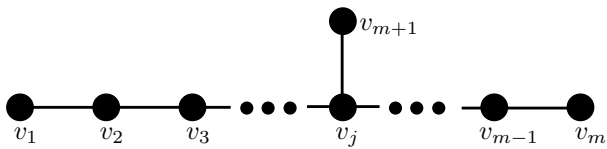
		\item For any tree $T$, let {\em $P_n(T)$} denote the Weinstein manifold obtained by plumbing $T^*S^n$ along the plumbing pattern $T$.  
		\item For $m \in \mathbb{N}$, let $Q_m^{2n}$ be the total space of an abstract Lefschetz fibration
		\begin{gather}
			\label{eqn Q_m^n}
			(A^{2n-2}_2; \alpha, \alpha, \cdots, \alpha, \beta, \beta),
		\end{gather}
		where the number of $\alpha$ in \eqref{eqn Q_m^n} is $m$. 
		We note that the notation $\alpha$ and $\beta$ are used in \eqref{eqn X} and \eqref{eqn Y}.
	\end{enumerate}
\end{definition}
\begin{remark}
	\label{rmk comparison of Q and Y}
	We note that $Q_{2k+1}^{2n+2} = \Y$.
\end{remark}

\begin{lemma}
	\label{lemma diffeomorphic families 1}
	If $n = 2$, then, $P_{n+1}(T_m^j)$ and $P_{n+1}(T_m^{j+2})$ are diffeomorphic to each other. 
	If $n \geq 4$ is even, then $P_{n+1}(T_m^j)$ and $P_{n+1}(T_m^{j+4})$ are diffeomorphic to each other. 
\end{lemma}
\begin{proof}
	In order to prove Lemma \ref{lemma diffeomorphic families 1}, we would like to use a Lefschetz fibration on $P_{n+1}(T_m^j)$. 
	We note that \cite[Theorem 11.3]{Lee} gives an algorithm producing a Lefschetz fibration for a plumbing space $P_n(T)$ for any $T$. 
	Thus, by applying \cite[Theorem 11.3]{Lee} to $P_{n+1}(T_m^j)$, we can produce a desired Lefschetz fibration whose total space is $P_{n+1}(T_m^j)$. 
	The resulting Lefschetz fibration is 
	\begin{gather}
		\label{eqn LF}
		(A_2^{2n}; \alpha, \alpha, \cdots, \alpha, \beta, \alpha, \cdots, \alpha, \beta),
	\end{gather}
	such that 
	\begin{itemize}
		\item the total number of $\alpha$ in \eqref{eqn LF} is $(m+1)$, and
		\item the first $\beta$ in \eqref{eqn LF} is located at $(j+1)^{th}$ position in the collection of vanishing cycles. 
	\end{itemize}

	In the current paper, we omit a detailed proof for the statement that the total space of \eqref{eqn LF} is $P_n(T_m^j)$.
	However, we prove the statement for a special case.
	The special case we consider is $P_{n+1}(T_3^2)$. 
	
	What we want to show is that the total space of 
	\[(A_2^{2n}; \alpha, \alpha, \beta, \alpha, \alpha, \beta)\]
	is $P_{n+1}(T_3^2)$. 
	Figure \ref{figure T_3^2} describes the base of the above Lefschetz fibration. 
	\begin{figure}[h]
		\centering
		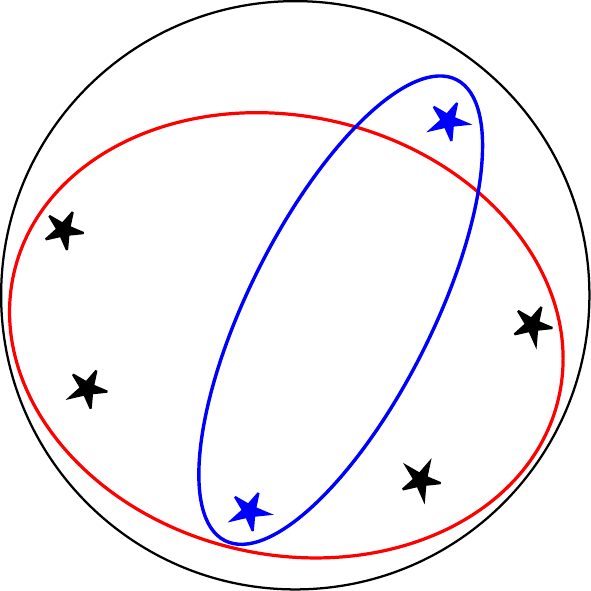		
		\caption{It describes the base of the given Lefschetz fibration. The black (resp.\  blue) star marks are singular values whose vanishing cycles are $\alpha$ (resp.\ $\beta$). The interior red (resp.\ blue) circle is the interior of sub-fibration defined on $W_1$ (resp.\ $W_2$).} 
		\label{figure T_3^2}
	\end{figure}
	For convenience, let $W$ denote the total space of the above abstract Lefschetz fibration. 
	
	We consider two submanifolds $W_1, W_2 \subset W$ by using the Lefschetz fibration on $W$.
	In order to define $W_1$, we consider the inverse image of the interior of the red circle in Figure \ref{figure T_3^2} under the Lefschetz fibration. 
	Roughly speaking, when we consider the Weinstein handle decomposition of $W$ corresponding to the abstract Lefschetz fibration, $W_1$ is a Weinstein domain obtained by `deleting' a critical handle from $W$. 
	Moreover, since $W_1$ could be seen as a total space of an abstract Lefschetz fibration 
	\[(A_2^{2n}; \alpha, \alpha, \beta, \alpha, \alpha),\]
	$W_1$ is equivalent to the Milnor fiber of $A_3$-type.  
	For the detail, see the proof of Lemma \ref{lemma strucutre of X}.
	
	Before defining $W_2$, we note that the fiber $A_2^{2n}$ can be seen as a union of two $T^*S^n$, or $T^*\alpha$ and $T^*\beta$ with the notation in Section \ref{section construction}. 
	Then, we define $W_2$ as the total space of the sub-fibration, whose fiber is $T^*\beta$, and whose base is the interior of the blue circle in Figure \ref{figure T_3^2}. 

	As similar to $W_1$, one could see $W_2$ as a union of corresponding Weinstein handles.
	Then, one could conclude that $W_2$ is a Weinstein domain equivalent to $T^*S^{n+1}$. 
	
	One could easily check that $W_1 \cup W_2$ is a neighborhood of the Lagrangian skeleton of $W$. 
	In order to obtain the Lagrangian skeleton, we use the Liouville structure which the abstract Lefschetz fibration induces. 
	It means that $W_1 \cup W_2$ is equivalent to $W$ up to symplectic completion. 
	Thus, it is enough to show that $W_1 \cup W_2$ is equivalent to $P_{n+1}(T_3^2)$. 
	
	We note that $P_{n+1}(T_3^2)$ is obtained by plumbing $T^*S^{n+1}$ to $P_{n+1}(A_3)$ where $A_3$ is the Dynkin diagram of $A_3$-type. 
	Since $W_1 \simeq P_{n+1}(A_3), W_2 \simeq T^*S^{n+1}$, it is enough to show that $W_1 \cup W_2$ is obtained by plumbing $W_1$ and $W_2$. 
	One can prove this easily by using the handle movements on $W_1$ and $W_2$.
	We omit the details, but instead we give Figure \ref{figure plumbing_base} describing a Lefschetz fibration type picture after handle slides.
	We note that Figure \ref{figure plumbing_base} is different from a general Lefschetz fibration picture, especially, the singular values are not lying on a circle whose center is the origin. 
	This is because we slide critical handles `onto' another critical handles.
	This is not allowed for a general Lefschetz fibration picture.
	\begin{figure}[h]
		\centering
		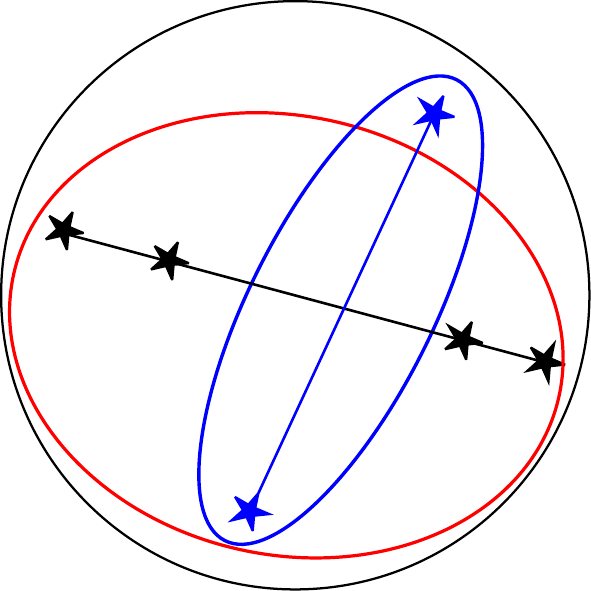		
		\caption{The black and blue curves are matching cycles.} 
		\label{figure plumbing_base}
	\end{figure}

	By applying the argument in Section \ref{section diffeo} to Lefschetz fibrations constructed above, one can prove Lemma \ref{lemma diffeomorphic families 1}. 
\end{proof}

The proof of Lemma \ref{lemma diffeomorphic families 1} also proves Lemma \ref{lemma diffeomorphic families 2}.

\begin{lemma}
	\label{lemma diffeomorphic families 2}
	If $n = 2$, then, $P_{n+1}(T_m^{m-1})$ and $Q_m^{2n+2}$ are diffeomorphic. 
	If $n \geq 4$ is even, $P_{n+1}(T_m^{m-3})$ and $Q_{m+1}^{2n+2}$ are diffeomorphic.
\end{lemma}
\begin{proof}
	By applying the same argument to the abstract Lefschetz fibrations in \eqref{eqn Q_m^n} and \eqref{eqn LF}, one can prove Lemma \ref{lemma diffeomorphic families 2}.
\end{proof}

From Lemmas \ref{lemma diffeomorphic families 1} and \ref{lemma diffeomorphic families 2}, we obtain diffeomorphic families such that some of whose members are well-studied plumbing spaces. 
The list of the families is given in Corollaries \ref{cor diffeomorphic list 1} and \ref{cor diffeomorphic list 2}. 

\begin{corollary}
\label{cor diffeomorphic list 1}
We have a following list of Weinstein manifolds which are diffeomorphic to each other, if their dimensions are $6= 2n+2$, i.e., $n=2$. 
\begin{itemize}
	\item {\em The Milnor fibers of $A_6$ and $E_6$-singularities are diffeomorphic}, since they are $P_{n+1}(T_5^1)$ and $P_{n+1}(T_5^3)$, respectively.
	\item {\em The Weinstein manifold $Q_6^{2n+2}$ and the Milnor fibers of $A_7$, $E_7$, and $D_7$-singularities are diffeomorphic}, since the Milnor fibers are $P_{n+1}(T_6^1), P_{n+1}(T_6^3)$, and $P_{n+1}(T_6^5)$, respectively.
	\item {\em The Milnor fibers of $A_8$ and $E_8$-singularities are diffeomorphic}, since they are $P_{n+1}(T_7^1)$ and $P_{n+1}(T_7^3)$, respectively.
	\item For any $m \geq 3$, {\em the Weinstein manifold $Q_{m+1}^{2n+2}$ and the Minor fiber of $D_{m+1}$-singularities are diffeomorphic}, since the Milnor fiber is $P_{n+1}(T_m^{m-1})$. 
	\item For $k \geq 2$, {\em the Weinstein manifold $Q_{2k+1}^{2n+2}$ and the Milnor fibers of $A_{2k+1}, D_{2k+1}$-singularities are diffeomorphic}, since the Milnor fibers are $P_{n+1}(T_{2k}^1)$ and $P_{n+1}(T_{2k}^{2k-1})$, respectively.
\end{itemize}
\end{corollary}

\begin{corollary}
	\label{cor diffeomorphic list 2}
	We have a following list of Weinstein manifolds which are diffeomorphic to each other if their dimensions are $2n+2 \geq 8$ with even $n \geq 4$. 
	\begin{itemize}
		\item {\em The Milnor fiber of $E_7$-singularity and $Q_7^{2n+2}$ are diffeomorphic}, since the Milnor fiber is $P_{n+1}(T_6^3)$.
		\item {\em The Milnor fibers of $A_8$ and $E_8$-singularities are diffeomorphic}, since the Milnor fibers are $P_{n+1}(T_7^1)$ and $P_{n+1}(T_7^5)$, respectively.  
		\item {\em The Milnor fiber of $A_{4k+1}$-singularity and $Q_{4k+1}^{2n+2}$ are diffeomorphic}, since the Milnor fiber is $P_{n+1}(T_{4k}^1)$.
		\item {\em The Milnor fiber of $D_{4k+2}$-singularity and $Q_{4k+2}^{2n+2}$ are diffeomorphic}, since the Milnor fiber is $P_{n+1}(T_{4k+1}^2)$. 
		\item {\em The Milnor fibers of $A_{4k+3}$ and $D_{4k+3}$-singularities are diffeomorphic}, since the Milnor fibers are $P_{n+1}(T_{4k+2}^1)$ and $P_{n+1}(T_{4k+1}^2)$, respectively.
	\end{itemize}
\end{corollary}

\begin{remark}
When we consider the Milnor fibers of simple singularities having dimension $2$, then we can check Corollaries \ref{cor diffeomorphic list 1} and \ref{cor diffeomorphic list 2}, without using Lefschetz fibrations.
Because of the dimension reason, we only need to compare their Euler characteristics and the numbers of boundary components.
For example, $A_7, D_7$ and $E_7$ have the same Euler characteristics, and the numbers of their boundary components are $2$. 
Thus, they are diffeomorphic to each other.

Moreover, it is also simple to compare them as Weinstein manifolds if the dimension is two.  
However, for the case of higher dimension, it would be not simple to compare them as Weinstein manifolds. 
In the next section, we will show that they are different as Weinstein manifolds in dimension $\geq 6$. 
\end{remark}

We note that we can have bigger diffeomorphic families from Lemmas \ref{lemma diffeomorphic families 1} and \ref{lemma diffeomorphic families 2}, but we only consider smaller families in Corollaries \ref{cor diffeomorphic list 1} and \ref{cor diffeomorphic list 2}.
This is because, in order to compare them as Weinstein manifolds, we need to compute their symplectic invariant.
For the named spaces which are contained in the smaller families, the computations are well-studied, and we would like to use the well-studied computations. 

For the other plumbing spaces which are not contained in the smaller families, we are working on comparing them as Weinstein manifolds.  

\subsection{Distinguishing symplectic cohomologies}
\label{subsection exotic families}
In this section, we compare the Weinstein structures of the listed Weinstein manifolds. 
We note that for this subsection, we do not need to separate the case of $n=2$ and the case of even $n \geq 4$. 
Moreover, even for the case of odd $n$, the results in this subsection hold.

Before proving Theorem \ref{thm exotic families}, we note that Milnor fibers of ADE-types can be described by Milnor fibers of invertible polynomials.
In other words, there is an invertible matrix $(A_W)_{ij}=a_{ij}$ so that we can write the Milnor fibers as $W^{-1}(1)$ where
\[W=\sum_{i=1}^{n+1} \prod_{j=1}^{n+1} x_j^{a_{ij}}.\]
If we define its transpose $W^T$ as a Milnor fiber of an invertible polynomial with the exponent matrix $A_W^T$, then we obtain the following list of polynomials for Milnor fibers of $ADE$-types and their transposes. 
\begin{itemize}
	\item $W_{A_m} = x_1^{m+1} + x_2^2+ x_3^2 +\cdots +x_{n+1}^2, \hskip 0.3cm W_{A_m}^T = W_{A_m},$
	\item $W_{D_m} = x_1^{m-1}+ x_1x_2^2 +x_3^2 \cdots + x_{n+1}^2,  \hskip 0.3cm W_{D_m}^T = x_1^{m-1}x_2 + x_2^2 +x_3^2 \cdots + x_{n+1}^2,$
	\item $W_{E_6}= x_1^4 + x_2^3 +x_3^2 +\cdots +x_{n+1}^2,\hskip 0.3cm W_{E_6}^T = W_{E_6},$
	\item $W_{E_7}= x_1^3 + x_1 x_2^3 +x_3^2 +\cdots +x_{n+1}^2, \hskip 0.3cm W_{E_7}^T = x_1^3 x_2 + x_2^3 +x_3^2 +\cdots +x_{n+1}^2,$
	\item $W_{E_8}= x_1^5 + x_2^3 +x_3^2 +\cdots +x_{n+1}^2, \hskip 0.3cm W_{E_8}^T = W_{E_8}.$
\end{itemize} 

With the above argument, \cite{Lekili-Ueda} gives Proposition \ref{prop SH computation}.

\begin{proposition}(Theorem 1.2 \cite{Lekili-Ueda})
	\label{prop SH computation}
	Let $W$ be a polynomial of $ADE$-type for $n \geq 2$. Then,
	\begin{enumerate} 
		\item we have an equivalence of categories 
		\[WFuk\left(W^{-1}(1)\right) \simeq \mathrm{MF}\left(W^T+x_0x_1\cdots x_{n+1}, \Gamma_{W^T}\right).\]
		The right handed side is a category of $\Gamma_{W^T}$-equivariant matrix factorizations. Here $\Gamma_{W^T}$ is a maximal group of abelian symmetries such that $W^T+\prod_{i=0}^{n+1}x_i$ becomes a semi-invariant, i.e.,
		\[\Gamma_{W^T}:= \left\{(\lambda_0, \ldots, \lambda_{n+1}): W^T(\lambda_1x_1, \ldots, \lambda_{n+1}x_{n+1})+\prod_{i=0}^{n+1}(\lambda_ix_i) = \lambda(W^T+\prod_{i=0}^{n+1}x_i) \right\}.\]
		\item $SH^*\left(W^{-1}(1)\right) \simeq HH^*\left(\mathrm{MF}\left(W^T+x_0x_1\cdots x_{n+1}, \Gamma_{W^T}\right)\right)$
	\end{enumerate}
	Also, one can compute $SH^*$ using the second isomorphism combined with results of \cite{Ballard-Favero-Katzarkov}. 
\end{proposition}

We do not prove Proposition \ref{prop SH computation}, but an interested reader can find a detailed computation in \cite[Section 5]{Lekili-Ueda}. 

\begin{proof}[Proof of Theorem \ref{thm exotic families}.]
We would like to point out that members in the families are Milnor fibers of $ADE$-types and $Q_m^{2n+2}$. 
The symplectic cohomology rings of Milnor fibers of simple singularities have been computed using homological mirror symmetry as mentioned in Proposition \ref{prop SH computation}. 
On the other hand, Lemma \ref{lemma wrapped Fukaya of Y} implies that
\[SH^*(Q^{2n}_{m+1}) \simeq SH^*(P_n(T^1_m)) \times SH^*(T^*S^n).\]
We note that $P_n(T_m^1)$ is the Milnor fiber of $A_{m+1}$-type. 
Thus, one can compute symplectic cohomologies of all Weinstein manifolds in Corollaries \ref{cor diffeomorphic list 1} and \ref{cor diffeomorphic list 2}.

From \cite{Lekili-Ueda}, we can check the following facts.
\begin{itemize}
	\item Milnor fibers of $A_k, D_k$ and $E_k$-types have different symplectic cohomologies, because the contributions from twisted sectors of each $HH^*(\mathrm{MF})$ are different from each other.
	\item  A symplectic cohomology carries extra information of weights. 
	For matrix factorizations, weights are coming from the weights of variables of $W^T$. 
	A mirror $\mathbb C^*$-action on a symplectic cohomology ring can be found in \cite{Seidel-Solomon}.
	The weights of the each generators of $SH^*$ are equal or less than their degrees. 
	There is only one family of generators (other then unit) of $SH^*$ such that whose weight and degree are the same.
	Every generator in the family has weight $n$ and degree $n$.
	\item In $SH^*(Q_m^{2n}) \simeq SH^*(P_n(T_m^1)) \times SH^*(T^*S^n)$, there is an element of degree $2n$ and weight $2n$. The element is a product of elements in $SH^*(P_n(T_m^1))$ and $SH^*(T^*S^n)$ whose degrees and weights are $n$. 
	Moreover, there are no such elements inside $SH^*(P_n(T_m^1)), SH^*(P_n(T_m^{m-1}))$, or $SH^*(E_7)$.
\end{itemize}

The above three prove that Weinstein manifolds in each families in Corollaries \ref{cor diffeomorphic list 1} and \ref{cor diffeomorphic list 2} have different symplectic cohomologies.
This completes the proof.  
\end{proof}

\bibliographystyle{amsalpha}
\bibliography{exotic}

\end{document}